\documentclass[review,hidelinks,onefignum,onetabnum]{siamart250211}

\usepackage{comment}

\newcommand{\blue}{\normalcolor}
\newcommand{\bluetext}[1]{#1}
\newcommand{\nc}{\normalcolor}

\usepackage{todonotes}
\newcommand{\ijeffx}[1]{}

\makeatletter%
\@mparswitchfalse%
\makeatother%
\normalmarginpar
\usepackage{cite}
\usepackage{lipsum}
\usepackage{amsfonts}
\usepackage{graphicx}
\usepackage{epstopdf}
\usepackage{algorithmic}
\usepackage{zymacros}
\usepackage[percent]{overpic}
\newcommand{\lgg}{\underline{g}}
\ifpdf
  \DeclareGraphicsExtensions{.eps,.pdf,.png,.jpg}
\else
  \DeclareGraphicsExtensions{.eps}
\fi

\newsiamremark{remark}{Remark}
\newsiamremark{hypothesis}{Hypothesis}
\crefname{hypothesis}{Hypothesis}{Hypotheses}
\newsiamthm{claim}{Claim}
\newsiamremark{fact}{Fact}
\crefname{fact}{Fact}{Facts}

\headers{Energy Approach from $\varepsilon$-Graph with Connectivity Functional}{Calder J., Hao W., Lee S., Yang Y.}

\title{Energy Approach from $\varepsilon$-Graph to Continuum Diffusion Model with Connectivity Functional\thanks{Submitted to the editors DATE.
\funding{S.L. and W.H. was supported by National Institute of General
Medical Sciences through grant 1R35GM146894. J.C. was supported by NSF-DMS:2436333, as well as an Albert and Dorothy Marden Professorship.}}}

\author{Yahong Yang\thanks{School of Mathematics, Georgia Institute of Technology, Atlanta, GA
(\email{yyang3194@gatech.edu}).}\and Sun Lee\thanks{Department of Mathematics, The Pennsylvania State University, State College,
PA
  (\email{wxh64@psu.edu}, \email{skl5876@psu.edu}).}\and Jeff Calder\thanks{School of Mathematics, 
University of Minnesota, Minneapolis, MN
  (\email{jwcalder@umn.edu}).}
\and Wenrui Hao\footnotemark[3] }

\usepackage{amsopn}

\ifpdf
\hypersetup{
  pdftitle={Energy Approach from $\varepsilon$-Graph to Continuum Diffusion Model with Connectivity Functional},
  pdfauthor={Calder J., Hao W., Lee S., Yang Y.}
}
\fi

\begin{document}

\maketitle

\begin{abstract}
We derive an energy--based continuum limit for $\varepsilon$-graphs endowed
with a general connectivity functional.  We prove that
the discrete energy and its continuum counterpart differ by at most
$\mathcal{O}(\varepsilon)$; the prefactor involves only the
$W^{1,1}$-norm of the connectivity density as $\varepsilon\to0$, so the
error bound remains valid even when that density has strong local
fluctuations.  As an application, we introduce a neural--network procedure
that reconstructs the connectivity density from edge--weight data and then
embed the resulting continuum model into a brain-dynamics framework.  In
this setting the usual constant diffusion coefficient is replaced by the
spatially varying coefficient produced by the learned density, yielding
dynamics that differ significantly from those obtained with conventional
constant-diffusion models.
\end{abstract}



\section{Introduction}

The human brain is an extraordinarily complex system whose function emerges from the dynamic interactions of a vast network of neurons. To study this complexity at a macroscopic level, brain activity is often represented through {\it parcellation} of neuroimaging data such as PET and fMRI, where the brain is divided into distinct regions of interest (ROIs). Each ROI is then treated as a node in a graph, and edges represent either structural connectivity (e.g., diffusion MRI tractography) or functional connectivity (e.g., correlations of fMRI time series) \cite{sporns2011human,craddock2012whole,schaefer2018local}.  
The topology of this connectome is highly organized, supporting efficient information transfer, a balance between functional segregation and integration, and robustness to perturbations. A powerful mathematical framework for analyzing such systems is the {\it geometric graph}, where nodes are embedded in a spatial domain (the brain volume or cortical surface) and edges are weighted according to measures of proximity or connectivity.  

A central challenge in computational neuroscience is to bridge discrete graph-based models of brain connectivity with continuum partial differential equation (PDE) models that capture large-scale spatiotemporal dynamics of neural activity. This challenge is particularly relevant in the context of neurodegenerative diseases, where PDE-based diffusion models have been employed to describe the propagation of pathological proteins such as amyloid-beta and tau in Alzheimer’s disease \cite{raj2012network,hao2016mathematical,hao2025optimal,zheng2022data,hao2022optimal}.  

In this work, we develop a rigorous energy-based framework to derive a continuum diffusion model directly from an $\varepsilon$-graph representation of the brain network, where edge weights are defined through a connectivity functional $g$. 
The associated distance metric is given by the shortest path between points, with path length weighted by $g(\vx)$, encoding the local ``cost'' or ``resistance'' to neural communication or molecular transport at location $\vx$.  
The central contribution of this paper is to show that the nonlocal Dirichlet energy defined on such a graph converges to a local continuum energy, thus providing a principled link between network-based and PDE-based descriptions of brain dynamics.

There is already a substantial body of work on the convergence from discrete models to continuum models, such as the convergence of the Cauchy–Born rule and the Peierls–Nabarro model in materials science \cite{ming2007cauchy,luskin2013atomistic,liu2007ab,makridakis2011priori,luo2018atomistic,yang2023stochastic,yang2022convergence}, as well as on convergence from graph-based structures to continuum models \cite{calder2018game,calder2020calculus,calder2022improved,calder2022lipschitz,slepcev2019analysis,garcia2020error,bungert2024convergence,trillos2025minimax}. The former concerns analysis in crystalline structures, where the atomistic model is defined on the lattice $\varepsilon \mathbb{Z}^d$. In contrast, the graph-based models are typically defined on random discrete spaces, and have typically arisen in the analysis of machine learning and data science algorithms in the large data limit. This includes works on continuum limits of semi-supervised learning based on the $p$-Laplacian \cite{calder2018game,slepcev2019analysis} and graph Poisson equations \cite{bungert2024convergence}, as well as continuum limits for the spectrum of the graph Laplacian (i.e., spectral convergence, see e.g., \cite{calder2022improved,calder2022lipschitz,garcia2020error,trillos2025minimax} and references therein), among many others. In this paper, we follow the latter framework of graph-based models.

Our framework for convergence from graph-based discrete energies to continuum local energies is based on the variational arguments outlined in several works \cite{calder2020calculus,garcia2020error,bungert2024convergence}, with modifications to account for a general connectivity functional on the graph. The proof can be divided into two steps by introducing an intermediate continuous nonlocal energy. The convergence from the discrete energy to the continuous nonlocal energy with connectivity density $g$ (Theorem~\ref{main1}) can be obtained via a Bernstein inequality for U-statistics, provided that $g$ admits a positive lower bound. For the convergence from the continuous nonlocal energy to the local energy, we need a result on locality of optimal paths, which is established in Lemma~\ref{two bound} and relies on the smoothness of the domain. In particular, when $\vx$ and $\vy$ are close, the optimal path between them remains concentrated near $\vx$ (Proposition~\ref{small path}); this is the key ingredient in proving the convergence from the continuous nonlocal energy to the local energy (Theorem~\ref{main2}). The error introduced by reducing the optimal path from $\vx$ to $\vy$ depends on the maximal derivative of the connectivity density $g$ around $\vx$. After averaging over the entire domain, the prefactor in the final error estimate depends on the $W^{1,1}$-norm of $g$, rather than its $W^{1,\infty}$-norm.

Based on the continuum energy models, one can derive the corresponding diffusion equations and perform numerical simulations. A main difficulty is that these equations involve the connectivity density function~$g$, whereas in most problems we only have access to the connectivity functional matrix (i.e., the edge weights of a graph) \cite{shahhosseini2022functional}. Thus, the first step is to recover the density function~$g$ from the data of edge weights. To this end, we approximate $g$ by a neural network~$g_{\vtheta}$. Specifically, given a candidate function~$g_{\vtheta}$, we compute the induced edge weights in the graph. Although evaluating these weights exactly requires solving an optimization problem for the optimal paths, our theoretical results show that the optimal path is localized near the endpoints~$\vx,\vy$. This allows us to apply a linear approximation of the path to estimate the edge weights. We then train the neural network using these estimated weights as data, thereby obtaining an approximation of~$g$. Our analysis verifies that this method is effective and achieves a linear approximation rate in sensitive and practical experiments, provided that the training error remains controlled away from boundary effects.

Applying the derived continuum model in simulations yields diffusion equations with spatially varying coefficients. In this setting, the standard constant diffusion coefficient, which is commonly assumed in practice \cite{murray2007mathematical,crank1979mathematics,osada2006homogenization}, is replaced by a spatially varying coefficient determined by the learned density. As a result, the dynamics differ substantially from those produced by constant-coefficient diffusion models.

\section{Weighted Graph Diffusion}
\subsection{\texorpdfstring{$\varepsilon$-Graph with Connectivity Density $g$}{epsilon-Graph with Connectivity Density g}}

Let {\blue \(\Omega\) be a closed and bounded domain}, and \(\vx_1, \dots, \vx_n\) be an i.i.d. sequence of random variables drawn from a probability density \(\rho:\Omega \to \mathbb{R}\) that is positive, bounded, and Lipschitz continuous.  We denote the set of graph vertices by
\[
\fX_n = \{\vx_1, \dots, \vx_n\}.
\]
We consider a random geometric graph whose edge weights are defined by
\[
w_{ij} = \eta\Bigl( \frac{d_g(\vx_i, \vx_j)}{\varepsilon}\Bigr),
\]
where \(\eta:[0,\infty)\to [0,\infty)\) is a positive, continuous, and decreasing function on \([0,1]\) that vanishes on \([1,\infty)\).

The metric \(d_g\) is defined as follows.

\begin{definition}\label{def:weighted_distance}
The weighted distance \(d_g:\Omega\times\Omega\to \mathbb{R}\) is given by
\begin{equation}\label{eq:weighted_distance}
d_g(\vx,\vy) = \inf_{\substack{\gamma(0)=\vx, \\ \gamma(1)=\vy, \\ \gamma \in C^1([0,1];\Omega)}}
\int_0^1 |\gamma'(t)|\, g\bigl(\gamma(t)\bigr)\,\D  t.
\end{equation}
That is, \(d_g(\vx,\vy)\) is the length of the shortest path from \(\vx\) to \(\vy\) in \(\Omega\), where the length is weighted by the connectivity density \(g\).
\end{definition}

The first variation of the functional leads to the Euler-Lagrange equation:
\begin{equation}
\frac{\mathrm{d}}{\mathrm{~d} t}\left(g(\gamma(t)) \frac{\gamma^{\prime}(t)}{\left|\gamma^{\prime}(t)\right|}\right)- \nabla g(\gamma(t))\left|\gamma^{\prime}(t)\right|=0 .\notag
\end{equation}
In practice, one can solve the associated Euler–Lagrange equation to find the optimal path between any two points $\vx$ and $\vy$. \blue Alternatively, 
one can solve the associated eikonal equation $|\nabla u| = g$ with boundary condition $u(\vx)=0$, and appropriate state constrained condition on $\partial \Omega$, whose solution is $u(\vy) = d_g(\vx,\vy)$. Optimal paths can then be computed by dynamic programming, for more details we refer to \cite{bardi1997optimal}.  Either way, computing the metric \(d_{g}(\vx,\vy)\) for every pair of nodes is prohibitive when \(n\) is large.  In this paper, we derive a continuum formulation whose cost is independent of \(n\); the resulting model depends only on the connectivity field \(g(\vx)\) and requires no pairwise distance evaluations. \nc




Given this metric, the corresponding Dirichlet energy of a function \(u:\fX_n\to\mathbb{R}\) is defined on the graph by
\begin{equation}\label{eq:dirichlet_energy}
E_n[u] = \frac{1}{\sigma_\eta \, n^2 \varepsilon^{d+2}} \sum_{i,j=1}^n w_{ij}\,\bigl(u(\vx_i)-u(\vx_j)\bigr)^2,
\end{equation}
where 
\[
\sigma_\eta = \int_{\mathbb{R}^d} \eta(|\vw|)\,|w_1|^2\,\D  \vw
\]
is a normalization constant related to \(\eta\). Similarly, the corresponding normalized graph Laplacian is given by
\begin{equation}\label{eq:graph_laplacian}
L_n[u](\vx_i) = \frac{2}{\sigma_\eta \, n \varepsilon^{d+2}} \sum_{j=1}^n w_{ij}\,\bigl(u(\vx_i)-u(\vx_j)\bigr).
\end{equation}
{\blue This graph-based description is widely used to simulate brain activity, as noted in the introduction.
In the present work, we start from such a graph representation and pass to a continuum limit, obtaining a PDE model for large-scale brain dynamics.
The resulting continuum formulation is both more faithful to physiological connectivity and more amenable to analysis than the discrete biological models commonly employed.}

\subsection{Preliminaries}\label{prelim}

In this subsection, we collect the assumption and proposition regarding the weighted distance function that will be used throughout the paper. The proof of Proposition \ref{small path} can be found in the supplementary materials.

\begin{assumption}\label{assump:connectivity_functional}
{\blue For the connectivity density function \(g:{\Omega}\subset\sR^d\to\mathbb{R}\)}, we assume that 
\[
{\blue \bar{g} := \sup_{\vy\in {\Omega}} g(\vy)<\infty
\quad \text{and} \quad
0<\underline{g} := \inf_{\vy\in {\Omega}} g(\vy) .}
\]\end{assumption}



\ijeffx{The proposition below is not true (in particular, the upper bound), since the paths are constrained to travel inside $\Omega$. It would hold if $B_r(\vx)\subset \Omega$ or $B_r(\vy)\subset \Omega$ for $r=|\vx-\vy|$. Otherwise, the result should be 
\begin{equation}\label{eq:distance_bounds0new}
\underline{g}|\vx-\vy| \le d_g(\vx,\vy) \le \bar{g}\,d_\Omega(\vx,\vy),
\end{equation}
where $d_\Omega(\vx,\vy) := d_1(\vx,\vy)$ is the geodesic distance on $\Omega$. It agrees with Euclidean distance when $\Omega$ is convex. There are approximate versions of the upper bound of the form
\begin{equation}
d_\Omega(\vx,\vy) \leq |\vx-\vy| + C|\vx-\vy|^{1+\alpha},
\end{equation}
for $|\vx-\vy|\leq r_\Omega$ provided the boundary $\partial \Omega$ is uniformly $C^{1,\alpha}$ (here, $r_\Omega$ depends only on the domain and its boundary regularity). The typical assumption is $C^{1,1}$ in which case the error term is quadratic. See Assumption 5.2 and Proposition 5.1 in \cite{bungert2023uniform}.}


\begin{proposition}\label{small path}
{\blue Suppose that $\Omega$ has a $C^{1,1}$ boundary, $g\in W^{1,\infty}(\Omega)$, $\rho\in L^1(\Omega)$,} and Assumption \ref{assump:connectivity_functional} holds. For any \(\varepsilon>0\), $\vx\in\Omega$ and
\(
\lambda=\frac{\bar{g}}{\lgg}+1
\) with $ \varepsilon< \lgg\min\{r_\Omega/2\lambda, 1/\lambda B\}$,
we have for any $\vx,\vy\in\Omega$ with $d(\vx,\vy)\le \varepsilon$,
\begin{align}\label{eq:small_path}
 \left|d_g(\vx,\vy)-g(\vx)|\vx-\vy|\right|\le \frac{\varepsilon^{2}}{\lgg^{2}}
\left[\,6\lambda\;\sup_{\va\in \Omega\cap B_{4\lambda\varepsilon/\lgg}(\vx)} |\nabla g(\va)| \;+\; B\bar g \right]
\end{align}
where $B$ and $r_\Omega$ are constants only dependent on $\Omega$.
\end{proposition}

This result shows that, if $\vx$ and $\vy$ are sufficiently close and the connectivity density $g$ is smooth, the weighted distance between them is well approximated by the straight‐line distance. Therefore, in Section~\ref{Method_Ex}, we simplify the distance calculation in our experiments by using the local Euclidean (straight‐line) distance.




\section{Energy Approach}
\subsection{Convergence from discrete energy to nonlocal energy}

In this subsection, we use the nonlocal energy as a bridge between the discrete energy and the continuum energy. First, we define the nonlocal energy functional:
\begin{equation}\label{eq:nonlocal_energy}
    I_\varepsilon[u] := \frac{1}{\sigma_\eta \varepsilon^d} \int_{\Omega}\int_{\Omega} \eta\Bigl( \frac{d_g(\vx,\vy)}{\varepsilon} \Bigr) \frac{\bigl(u(\vx)-u(\vy)\bigr)^2}{\varepsilon^2}\,\rho(\vx)\rho(\vy)\D \vy\D \vx,
\end{equation}
where the normalization constant is given by
\[
\sigma_\eta = \int_{\mathbb{R}^d} \eta(|\vw|)\,|w_1|^2\D \vw.
\]

Our goal is to show that
\(
    \left|E_{n}[u]-I_{\varepsilon}[u]\right|
\)
is small. The idea of the proof is based on Bernstein's inequality for U-statistics, which implies that when the number of sample points is sufficiently large, the corresponding U-statistics will converge in probability to their expectation.




\begin{theorem}\label{main1}
     Suppose that Assumption \ref{assump:connectivity_functional} holds. For any $0<\delta\le1$ and Lipschitz continuous $u$ we have that 
     \[
\left|E_{n}[u]-I_{\varepsilon}[u]\right| \leq C \operatorname{Lip}[u]^2\left(\delta+\frac{1}{n}\right).
\]
holds with probability at least $1-2 \exp \left(-c n \varepsilon^d \delta^2\right)$, where $c,C$ are constants dependent on $\lgg$ and $\rho$.
\end{theorem}

\begin{proof}
The result follows directly from Bernstein’s inequality for U‐statistics (see \cite{arcones1995bernstein}); a complete argument appears in \cite[Lemma 5.28]{calder2020calculus} and \cite[Lemma 3.6]{bungert2024convergence}. The only adjustment here is the use of Assumption \ref{assump:connectivity_functional} to control the support of our kernel. Indeed, by Assumption \ref{assump:connectivity_functional} we have
\[
\{(\vx,\vy)\mid d_g(\vx,\vy)\le \varepsilon\}
\;\subset\;
\{(\vx,\vy)\mid \underline{g}\,\lvert\vx-\vy\rvert \le \varepsilon\},
\]then the support of $\eta\Bigl( \frac{d_g(\vx,\vy)}{\varepsilon} \Bigr)$ for each $\vx$ is the subset of $B_{ \frac{\varepsilon}{\lgg}}(\vx)$.
The remainder of the proof is identical to that in \cite[Lemma 5.28]{calder2020calculus}. We show the details in the supplementary materials.
\end{proof}

\subsection{Convergence from nonlocal energy to local energy}
Now, we estimate the error between the local energy and the nonlocal energy. First, we define the \textit{local energy} as
\begin{equation}\label{eq:local_energy}
I[u] := \int_{\Omega} \rho^2(\vx) \frac{|\nabla u(\vx)|^2}{g^{d+2}(\vx)}\D \vx.
\end{equation}
The idea for calculating the error is based on a first-principles calculation. In particular, one expands \(u(\vy)-u(\vx)\) via a Taylor expansion about \(\vx\) and then compares the resulting integrals with those defining the nonlocal energy \(I_\varepsilon[u]\). Proposition \ref{small path} allows us to control the remainder terms precisely.

The final estimate shows that the gap between the non-local energy
\(I_{\varepsilon}[u]\) and its local counterpart \(I[u]\) is first-order in
\(\varepsilon\); namely, there exists a constant \(C^{*}>0\) such that
\[
  \bigl|\,I_{\varepsilon}[u]-I[u]\,\bigr|
  \;\le\;C^{*}\,\varepsilon .
\]
The constant \(C^{*}\) depends on the connectivity density \(g\) only
through the quantity
\[
  \int_{\Omega}
     \sup_{\va\in\Omega\cap B_{3\lambda\varepsilon/\lgg}(\vx)}
         |\nabla g(\va)|\,
     \rho(\vx)\,d\vx,
\]
rather than on the full \(W^{1,\infty}(\Omega)\)-norm of \(g\), as
\(\varepsilon\to0\) it converges to the \(W^{1,1}(\Omega)\)-seminorm
(Proposition~\ref{W11}).


\begin{theorem}\label{main2}
Suppose that $\Omega$ has $C^{1,1}$ boundary, Assumption~\ref{assump:connectivity_functional} holds, $g\in W^{1,\infty}(\Omega)$, and $\rho,\eta$ are Lipschitz continuous with Lipschitz constants $\alpha$ and $\mu$. Let $u\in C^{1,1}(\Omega)$ with
\(
\beta = \|u\|_{C^{1,1}(\Omega)}.
\)
Then, for any $\varepsilon$ with $0< \varepsilon< \lgg\min\{r_\Omega/2\lambda, 1/\lambda B\}$, the error between the nonlocal energy $I_\varepsilon[u]$ and the local energy $I[u]$ is estimated by \begin{align}
  \bigl|I_\varepsilon[u] - I[u]\bigr| \le C_* \varepsilon.  
\end{align}Here
\begin{align}
C_*= &C\left(\alpha^2\beta\,\operatorname{Lip}(u) + \alpha^2\left(1+\frac{1}{\lgg^{d+3}}\right)\,\operatorname{Lip}^2(u)\right)\notag\\&+\frac{\alpha(\operatorname{Lip}(u))^2\mu V_d(1)}{\sigma_\eta\lgg^{d+4}} \left[\,6\lambda\;\int_\Omega\sup_{\va\in \Omega\cap B_{4\lambda\varepsilon/\lgg}(\vx)} |\nabla g(\va)| \;\rho(\vx)\,\D  \vx+\; B\bar g \right]\notag
\end{align}
where $\lambda,B,\lgg$ are defined in Assumption \ref{assump:connectivity_functional} and Proposition \ref{small path},
\(C>0\) depends only on the domain \(\Omega\) and \(\sigma_\eta\), $\operatorname{Lip}(u)$ denotes the Lipschitz constant of $u$, $V_d(1)$ is the volume of the unit ball in $\mathbb{R}^d$, and $\sigma_\eta>0$ is a constant associated with the kernel $\eta$.
\end{theorem}

\begin{proof}
For any $\vx,\vy\in\Omega$ satisfying 
\[
d(\vx,\vy) \le \varepsilon,
\]
we first apply the Taylor expansion:
\[
u(\vy) - u(\vx) = \nabla u(\vx)\cdot (\vy-\vx) + \mathcal{O}(\beta\,\varepsilon^2).
\]
It follows that
\[
\bigl(u(\vy) - u(\vx)\bigr)^2 = \bigl(\nabla u(\vx)\cdot (\vy-\vx)\bigr)^2 + \mathcal{O}\bigl(\beta\,\operatorname{Lip}(u)\,\varepsilon^3\bigr).
\]
Next, we have
\(
\rho(\vy) = \rho(\vx) + \mathcal{O}(\alpha\,\varepsilon).
\)

Recalling the definition of the nonlocal energy
\[
I_\varepsilon[u] = \frac{1}{\sigma_\eta\,\varepsilon^d} \int_\Omega\int_\Omega \eta\Bigl(\frac{d_g(\vx,\vy)}{\varepsilon}\Bigr)
\frac{(u(\vx)-u(\vy))^2}{\varepsilon^2} \rho(\vx)\rho(\vy)\,\D  \vy\,\D  \vx,
\]
we substitute the above approximations to obtain
\begin{align}\label{11}
I_\varepsilon[u]
&=\frac{1}{\sigma_\eta\,\varepsilon^{d+2}} \int_\Omega\int_\Omega \eta\Bigl(\frac{d_g(\vx,\vy)}{\varepsilon}\Bigr)
\Bigl[\bigl(\nabla u(\vx)\cdot (\vy-\vx)\bigr)^2 + \mathcal{O}\bigl(\beta\,\operatorname{Lip}(u)\,\varepsilon^3\bigr)\Bigr] \\[1mm]
&\hspace{2.5in}\times \Bigl[\rho(\vx) + \mathcal{O}(\alpha\,\varepsilon)\Bigr]\rho(\vx)\,\D  \vy\,\D  \vx \nonumber\\[1mm]
&=\frac{1}{\sigma_\eta\,\varepsilon^{d+2}} \int_\Omega\int_\Omega \eta\Bigl(\frac{d_g(\vx,\vy)}{\varepsilon}\Bigr)
\bigl(\nabla u(\vx)\cdot (\vy-\vx)\bigr)^2\rho^2(\vx)\,\D  \vy\,\D  \vx \nonumber\\[1mm]
&\hspace{2.5in} +\mathcal{O}\Bigl((\alpha^2\beta\,\operatorname{Lip}(u) + \alpha^2\,\operatorname{Lip}^2(u))\,\varepsilon\Bigr). \notag
\end{align}

We now analyze the leading term. We can decompose the distance as
\begin{align}\label{eq:dg_decomp}
d_g(\vx,\vy) =  g(\vx)|\vx-\vy| + e(\vx,\vy),\notag
\end{align} where the error term $e(\vx,\vy)$ collects the variations of $g$ along the optimal path, which can be bounded by results in Proposition \ref{small path}:\begin{align*}
    |e(\vx,\vy)|\le e^*(\vx):=\frac{\varepsilon^{2}}{\lgg^{2}}
\left[\,6\lambda\;\sup_{\va\in \Omega\cap B_{4\lambda\varepsilon/\lgg}(\vx)} |\nabla g(\va)| \;+\; B\bar g \right]
\end{align*}for any $d(\vx,\vy)\le \varepsilon$.

{\blue Thus, we write
\begin{align}
&\frac{1}{\sigma_\eta\,\varepsilon^{d+2}} \int_\Omega\int_\Omega \eta\Bigl(\frac{d_g(\vx,\vy)}{\varepsilon}\Bigr)
\bigl(\nabla u(\vx)\cdot (\vy-\vx)\bigr)^2\rho^2(\vx)\,\D  \vy\,\D  \vx \nonumber\\[1mm]
&=\frac{1}{\sigma_\eta\,\varepsilon^{d+2}} \int_\Omega\int_{\Omega\cap B_{\frac{\varepsilon}{\lgg}}(\vx)}
\eta\Bigl(\frac{g(\vx)|\vx-\vy|+ e(\vx,\vy)}{\varepsilon}\Bigr)
\bigl(\nabla u(\vx)\cdot (\vy-\vx)\bigr)^2\rho^2(\vx)\,\D  \vy\,\D  \vx \nonumber\\[1mm]
&=\frac{1}{\sigma_\eta\,\varepsilon^{d+2}} \int_\Omega\int_{\Omega\cap B_{\frac{\varepsilon}{\lgg}}(\vx)}
\eta\Bigl(\frac{g(\vx)|\vx-\vy|}{\varepsilon}\Bigr)
\bigl(\nabla u(\vx)\cdot (\vy-\vx)\bigr)^2\rho^2(\vx)\,\D  \vy\,\D  \vx \nonumber\\[1mm]
&\quad +\fO\left( \frac{1}{\sigma_\eta\,\varepsilon^{d+3}} \int_\Omega\int_{\Omega\cap B_{\frac{\varepsilon}{\lgg}}(\vx)}
\mu\,e^*(\vx)
\bigl(\nabla u(\vx)\cdot (\vy-\vx)\bigr)^2\rho^2(\vx)\,\D  \vy\,\D  \vx\right),\label{mainterm}
\end{align}}
where $\mu = \|\eta\|_{C^1(\mathbb{R})}$ and we have used a first-order Taylor expansion in the argument of $\eta$, and the support of $\eta\Bigl( \frac{d_g(\vx,\vy)}{\varepsilon} \Bigr)$ for each $\vx$ is the subset of $B_{ \frac{\varepsilon}{\lgg}}(\vx)$. 

To simplify the inner integrals, we choose an orthogonal matrix $\vA\in\mathbb{R}^{d\times d}$ so that
\[
\vA\,\nabla u(\vx) = |\nabla u(\vx)|\,\ve_1, \quad \text{with } \ve_1 = (1,0,\dots,0).
\]
Perform the change of variables
\[
\vz = \vx + \vA(\vy-\vx).
\]
Since $\vA$ is orthogonal, we have $|\vx-\vy| = |\vx-\vz|$ and
\[
\nabla u(\vx)\cdot (\vy-\vx)= \vA\nabla u(\vx)\cdot \vA(\vy-\vx)=|\nabla u(\vx)|\,\ve_1\cdot(\vz-\vx)= |\nabla u(\vx)|(z_1 - x_1).
\]
Then the first term in \eqref{mainterm} can be written as
\[
\frac{1}{\sigma_\eta\,\varepsilon^{d+2}} \int_\Omega |\nabla u(\vx)|^2 \rho^2(\vx)
\int_{B_{\frac{\varepsilon}{\lgg}}(\vx)\cap V} \eta\Bigl(\frac{g(\vx)|\vx-\vz|}{\varepsilon}\Bigr) |z_1-x_1|^2\,\D  \vz\,\D  \vx,
\]
where 
\[
V = \vx + \vA(\Omega-\vx).
\]
If $\operatorname{dist}(\vx,\partial\Omega) \ge \varepsilon/\lgg$, a change of variables $g(\vx)(\vz-\vx) = \varepsilon \,\vw$ shows that
\begin{align}
\int_{B_{\frac{\varepsilon}{\lgg}}(\vx)\cap V} \eta\Bigl(\frac{g(\vx)|\vx-\vz|}{\varepsilon}\Bigr) |z_1-x_1|^2\,\D  \vz
&=\frac{\varepsilon^{d+2}}{g(\vx)^{d+2}}\int_{B_1(\mathbf{0})} \eta(|\vw|)|w_1|^2\,\D  \vw=\frac{\varepsilon^{d+2}\sigma_\eta}{g(\vx)^{d+2}},
\end{align}where we recall that
{\blue \(
\sigma_\eta = \int_{B_1(\mathbf{0})} \eta(|\vw|)\,|w_1|^2\,\mathrm{d}\vw.
\)}
A similar bound holds for all $\vx\in\Omega$, i.e., \[\int_{B_{\frac{\varepsilon}{\lgg}}(\vx)\cap V} \eta\Bigl(\frac{g(\vx)|\vx-\vz|}{\varepsilon}\Bigr) |z_1-x_1|^2\,\D  \vz
\le\frac{\varepsilon^{d+2}\sigma_\eta}{g(\vx)^{d+2}}.\]Therefore, denote that
{\blue $\partial\Omega_{\varepsilon}:=\{\vx\in\Omega\mid \operatorname{dist}(\vx,\partial\Omega) < \varepsilon/\lgg\}$}, we have that \begin{align}
    &\left|I[u]-\frac{1}{\sigma_\eta\,\varepsilon^{d+2}} \int_\Omega\int_{\Omega\cap B_{\frac{\varepsilon}{\lgg}}(\vx)}
\eta\Bigl(\frac{g(\vx)|\vx-\vy|}{\varepsilon}\Bigr)
\bigl(\nabla u(\vx)\cdot (\vy-\vx)\bigr)^2\rho^2(\vx)\,\D  \vy\,\D  \vx\right|\notag\\\le&\left|\frac{1}{\sigma_\eta\,\varepsilon^{d+2}} \int_{\partial\Omega_{\varepsilon}}\int_{\Omega\cap B_{\frac{\varepsilon}{\lgg}}(\vx)}
\eta\Bigl(\frac{g(\vx)|\vx-\vy|}{\varepsilon}\Bigr)
\bigl(\nabla u(\vx)\cdot (\vy-\vx)\bigr)^2\rho^2(\vx)\,\D  \vy\,\D  \vx\right|\notag\\&+\left|\int_{\partial\Omega_{\varepsilon}} \rho^2(\vx) \frac{|\nabla u(\vx)|^2}{g^{d+2}(\vx)}\D \vx\right|\notag\\\le &2\left|\int_{\partial\Omega_{\varepsilon}} \rho^2(\vx) \frac{|\nabla u(\vx)|^2}{g^{d+2}(\vx)}\D \vx\right|\le 2C_\Omega\frac{\alpha^2\operatorname{Lip}(u)^2}{\lgg^{d+3}}\varepsilon,\label{12}
\end{align}where \(C_\Omega\) depends only on \(\Omega\). 
Here we used the fact that if \(\partial\Omega\) is of class \(C^{1,1}\), then the tubular neighborhood 
\(\partial\Omega_{\varepsilon}=\{x\in\Omega:\operatorname{dist}(x,\partial\Omega)<\varepsilon/\lgg\}\) 
satisfies \(|\partial\Omega_{\varepsilon}|\le C_\Omega\varepsilon/\lgg\) for small \(\varepsilon>0\).

It remains to estimate the remainder term
\[
R := \frac{1}{\sigma_\eta\,\varepsilon^{d+3}} \int_\Omega\int_{\Omega\cap B_{\frac{\varepsilon}{\lgg}}(\vx)} \mu\,e^*(\vx)
\bigl(\nabla u(\vx)\cdot (\vy-\vx)\bigr)^2 \rho^2(\vx)\,\D  \vy\,\D  \vx.
\]
The estimate of this part can be obtained from Proposition \ref{small path}. \begin{align}
    R\le& \frac{\alpha(\operatorname{Lip}(u))^2\mu}{\sigma_\eta\varepsilon^{d-1}\lgg^4}\int_\Omega\int_{\Omega\cap B_{\frac{\varepsilon}{\lgg}}(\vx)} \left[\,6\lambda\;\sup_{\va\in \Omega\cap B_{4\lambda\varepsilon/\lgg}(\vx)} |\nabla g(\va)| \;+\; B\bar g \right] \rho(\vx)\,\D  \vy\,\D  \vx\notag\\\le&\varepsilon\frac{\alpha(\operatorname{Lip}(u))^2\mu V_d(1)}{\sigma_\eta\lgg^{4+d}}\int_\Omega \left[\,6\lambda\;\sup_{\va\in \Omega\cap B_{4\lambda\varepsilon/\lgg}(\vx)} |\nabla g(\va)| \;+\; B\bar g \right] \rho(\vx)\,\D  \vx\label{Rbound}
\end{align}

By combining the bounds in \eqref{11}, \eqref{12}, and \eqref{Rbound}, we arrive at the desired conclusion.
\end{proof}

In the estimate for $\lvert I_\varepsilon[u] - I[u]\rvert$, the coefficient of the first term, 
\[
C\left(\alpha^2\beta\,\operatorname{Lip}(u) + \alpha^2\left(1+\frac{1}{\lgg^{d+3}}\right)\,\operatorname{Lip}^2(u)\right),
\] 
follows from a first–principles expansion and matches the analysis in \cite{calder2020calculus}, 
while the coefficient of the second term, 
\[
\fO\!\left(\int_{\Omega}\sup_{\va\in \Omega\cap B_{\frac{4\lambda\varepsilon}{\underline g}}(\vx)}
|\nabla g(\va)|\,\rho(\vx)\,\D\vx\right),
\]
arises from approximating the optimal path by a straight segment, which converges to the 
$W^{1,1}$–seminorm of $g$.

\begin{proposition}\label{W11}
Let $g\in W^{1,\infty}(\Omega)$ and assume $|\nabla g|$ is continuous at almost every point of $\Omega$.  
Let $\rho\in C(\Omega)$ be bounded. Then
\[
  \lim_{\varepsilon\to0}
  \int_{\Omega}
     \sup_{\va\in\Omega\cap B_{4\lambda\varepsilon/\lgg}(\vx)}
        |\nabla g(\va)|\,\rho(\vx)\,\D\vx
  \;=\;
  \int_{\Omega}|\nabla g(\vx)|\,\rho(\vx)\,\D\vx .
\]
\end{proposition}

\begin{proof}
Since \(g\in W^{1,\infty}(\Omega)\), we have
\(0\le|\nabla g|\le \|g\|_{W^{1,\infty}}\) almost everywhere.
Hence the integrand is dominated by
\(\|g\|_{W^{1,\infty}}\,\rho(\vx)\in L^{1}(\Omega)\).
For almost every \(\vx\in\Omega\) the ball
\(B_{4\lambda\varepsilon/\lgg}(\vx)\) shrinks to \(\{\vx\}\) as
\(\varepsilon\to0\), so the supremum converges to
\(|\nabla g(\vx)|\) due to that \(|\nabla g|\) is continuous at almost every point of \(\Omega\).
Pointwise convergence together with the uniform bound allows us to use the dominated convergence theorem, giving the desired limit.
\end{proof}


\subsection{Continuous Diffusion Equation}

We have obtained the continuum energies derived from the \(\varepsilon\)-graph. In this section, we briefly discuss the variational formulation of the obtained continuum energies.

Recall that the continuum energy for is given by
\begin{equation}\label{eq:iso_energy}
I[u] := \int_{\Omega} \rho^2(\vx) \frac{|\nabla u(\vx)|^2}{g^{d+2}(\vx)}\D\vx.
\end{equation}
An equilibrium state is characterized by the vanishing of the first variation of \(I[u]\). Formally, setting the Euler--Lagrange derivative \(\frac{\delta I[u]}{\delta u}=0\) leads to
\begin{equation}\label{eq:iso_equilibrium}
\nabla\cdot\!\left(\frac{\rho^2(\vx)}{g^{d+2}(\vx)}\,\nabla u(\vx)\right) = 0.
\end{equation}
For the dynamics, consider the gradient flow associated with \(I[u]\); that is, 
\[
\frac{\partial u}{\partial t} = -\frac{1}{2}M\,\frac{\delta I[u]}{\delta u},
\]
where \(M>0\) denotes the mobility. This formally yields the evolution equation
\begin{equation}\label{eq:iso_flow}
\frac{\partial u}{\partial t} = -M\,\nabla\cdot\!\left(\frac{\rho^2(\vx)}{g^{d+2}(\vx)}\,\nabla u(\vx)\right).
\end{equation}

The rigorous convergence of solutions from the discrete energy to the continuum energy, as well as a detailed analysis of the resulting diffusion equations, will be addressed in future work.

\section{Numerical Results}
We put into a more general setup, a reaction–diffusion framework is that many biological processes in Alzheimer’s disease can be naturally described in this way. The diffusion term captures spatial spreading in the brain, such as the propagation of misfolded proteins like amyloid- $\beta$ or $\tau$ along neural pathways, while the reaction term represents local biochemical processes including aggregation, clearance, or enzymatic degradation. 

More specifically, we consider the following nonlinear reaction-diffusion equation defined on a brain-shaped domain:
\begin{equation} \label{eq:model}
\begin{cases}
\displaystyle u_{t} - \nabla \cdot (D(\vx) \nabla u) = C (1 - u) u, & \quad (\vx, t) \in \Omega \times (0, T), \\[2ex]
\displaystyle \frac{\partial u}{\partial \vn} = 0, & \quad (\vx, t) \in \partial \Omega \times (0, T), \\[2ex]
u(\vx, 0) = u_{0}(\vx), & \quad \vx \in \Omega,
\end{cases}
\end{equation}
where $u = u(\vx, t)$ represents the state variable of interest, $D(\vx)$ is a spatially varying diffusion coefficient, $C \ge 0$ is a reaction parameter, and $\Omega \subset \mathbb{R}^3$ denotes the computational domain corresponding to a brain geometry. The homogeneous Neumann boundary condition ensures no-flux across the boundary $\partial \Omega$.

When $C = 0$, equation~\eqref{eq:model} reduces to the classical heat equation with spatially heterogeneous diffusion:
\[
u_t - \nabla \cdot (D(\vx) \nabla u) = 0.
\]
This case corresponds to pure diffusion without reaction and is mathematically equivalent to the equilibrium equation. In this context, the diffusion process models transport in a medium with no production or depletion.

To faithfully model realistic brain dynamics, the domain $\Omega$ is reconstructed from MRI data. Specifically, we use a 3D brain image dataset consisting of $91 \times 109 \times 91$ voxels, representing the spatial resolution in the $x$-, $y$-, and $z$-directions, respectively. Each voxel contains intensity information corresponding to a specific brain region. As an illustrative example, Figure~\ref{fig:brain-3d} shows a axial slice of 3D MRI image.

\begin{figure}[ht]
\centering
\includegraphics[width=0.32\textwidth]{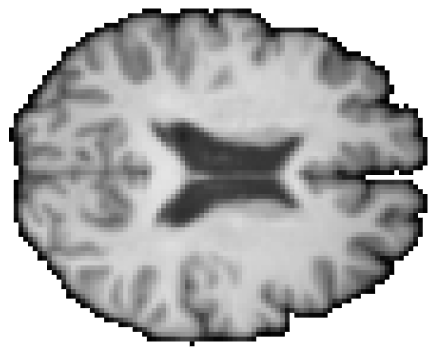}
\includegraphics[width=0.3\textwidth]{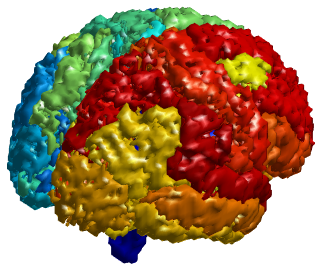}
\includegraphics[width=0.32\textwidth]{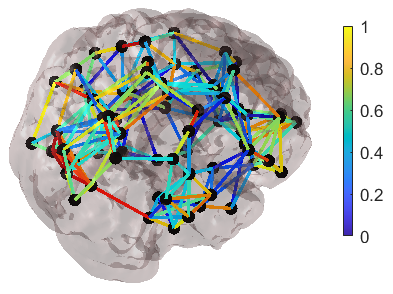}
\caption{Left: Axial slice from the 3D MRI volume used to construct the brain domain $\Omega$. Middle: Brain parcellation result.
Right: Visualization of functional connectivity (FC) between brain parcels. Each parcel is represented by its centroid, and FC is shown as lines connecting parcel pairs. Line colors indicate connection strength.}
\label{fig:brain-3d}
\end{figure}

In our modeling framework, the spatially varying diffusion coefficient $D(\vx)$ is not predefined, but rather learned from data derived from brain connectivity patterns. To achieve this, we employ a brain parcellation scheme that divides the brain domain into $68$ anatomically meaningful regions, known as parcels. Each parcel represents a distinct subregion of the brain.

Functional connectivity (FC) refers to the statistical dependency between activity patterns of different brain parcels, typically computed from resting-state fMRI signals \cite{shahhosseini2022functional}. The FC matrix encodes the strength of these inter-parcel relationships and serves as a proxy for functional communication pathways.

Rather than using all pairwise FC values—which may include noisy or weak long-range connections—we restrict attention to local functional connectivity between spatially adjacent or nearby parcels. Specifically, instead of fixing a distance cutoff, we adapt the neighborhood threshold according to the number of parcels $n$. More precisely, we set
\begin{equation}
\varepsilon = C \Big(\tfrac{\log n}{n}\Big)^{1/d}
\quad \text{or} \quad
\varepsilon = C \Big(\tfrac{\log n}{n}\Big)^{1/(d+2)}, \label{varepsilon}
\end{equation}
where $d$ is the spatial dimension, $C>0$ is a constant, and $n$ is the total number of parcels. The particular forms of the scaling in \eqref{varepsilon} are motivated by Theorem~\ref{main1}, which provides the theoretical basis for such choices. In particular, if we choose $\delta=\eps$ in order to obtain an $O(\eps)$ rate from Theorem~\ref{main1}, then the concentration estimate requires $n\eps^{d+2} \geq C\log n$ to ensure that the probability is larger than $1-2/n^p$, where $C$ depends on $p$. This condition yields a lower bound of the form $\eps \geq C(\log n/n)^{1/(d+2)}$. Another common scaling choice is $\eps \sim (\log n / n)^{1/d}$, but in this case we obtain a weaker probability guarantee.

This focus allows for more stable learning of biologically relevant diffusivity patterns. In total, we use 227 local FC values among the 68 parcels to inform and train the diffusion function \( D(\vx) \), aligning it with the underlying network structure of the brain. More details are provided in Section~\ref{Method_Ex}, where the diffusivity is modeled via the relation
\[
\frac{1}{g_\theta(\vx)^{d+2}} = D(\vx).
\]

Figure~\ref{fig:brain-3d} illustrates the brain parcellation and the corresponding subset of functional connectivity used in the model. These components provide the structural foundation for learning a heterogeneous, data-driven diffusion coefficient.

\subsection{Learning Spatially Varying Diffusivity via Neural Networks}\label{Method_Ex}

Recall that the equilibrium equation is given by
\begin{equation}\label{eq:iso_equilibrium1}
\nabla\cdot\!\left(\frac{\rho^2(\vx)}{g^{d+2}(\vx)}\,\nabla u(\vx)\right) = 0.
\end{equation}
To determine the solution of \eqref{eq:iso_equilibrium1}, we must specify the functions \(\rho(\vx)\) and \(g(\vx)\).

For our experiments, we assume that the underlying distribution of the data is uniform; that is, we take \(\rho(\vx)\) to be a constant function (i.e. $\rho(x)=1$). With a constant density, the remaining challenge is to select or estimate an appropriate form for the connectivity functional \(g(\vx)\), which encodes the local geometry of the domain and influences the diffusion process.

In the following, we describe how \(g(\vx)\) is obtained and demonstrate the resulting behavior of the diffusion model. We approximate \(g\) using a parametric model \(g_{\vtheta}:\Omega\to(0,\infty)\), where \(\vtheta\) denotes the parameters of the approximator. The approximator can be chosen as a polynomial (by the Weierstrass approximation theorem), a rational function \cite{nakatsukasa2018aaa}, or via neural network approximations \cite{yang2024near,yang2023nearly1,yang2023nearly}.

Suppose that our data is given by the connectivity matrix \(\{w_{ij}(\vx_i,\vx_j)\}_{(i,j)\in \sA}\). We first employ a numerical method to approximate the optimal path between any pair \(\vx,\vy\). Denote this optimal path by \(\gamma_{\vtheta}[\vx,\vy](t)\), which satisfies the following ODE:
\[
\frac{\mathrm{d}}{\mathrm{d}t}\Bigl(g_{\vtheta}(\gamma(t))\,\frac{\gamma'(t)}{|\gamma'(t)|}\Bigr) - \nabla g_{\vtheta}(\gamma(t))\,|\gamma'(t)| = 0,
\]
with the boundary conditions
\[
\gamma_{\vtheta}[\vx,\vy](0) = \vx \quad \text{and} \quad \gamma_{\vtheta}[\vx,\vy](1) = \vy.
\]
Then, the length of the optimal path is computed as
\[
w[\vx,\vy](\vtheta) := \int_0^1 |\gamma_{\vtheta}[\vx,\vy]'(t)|\,g_{\vtheta}\Bigl(\gamma_{\vtheta}[\vx,\vy](t)\Bigr) \, \mathrm{d}t.
\]
Subsequently, we determine the optimal parameters by solving
\begin{equation}\label{eq:loss_function}
\min_{\vtheta} L(\vtheta) := \min_{\vtheta} \frac{1}{|\sA|} \sum_{(i,j)\in \sA} \Biggl|\eta\Biggl(\frac{w[\vx_i,\vx_j](\vtheta)}{\varepsilon}\Biggr) - w_{ij}\Biggr|^2,
\end{equation}
where \(\eta\) is a fixed function and \(\varepsilon\) is a small parameter. Here, the index set \(\sA\) is chosen so that the pairs \(\vx,\vy\) are sufficiently close and the straight line connecting them lies within \(\Omega\).

The training procedure described above is rather complex, particularly due to the optimization in the optimal path. For each pair \(\vx,\vy\), we must solve the optimal path problem, and the parameter \(\vtheta\) remains unknown. Moreover, to train \eqref{eq:loss_function} we need an explicit expression for \(w[\vx,\vy](\vtheta)\).

Assume \(g_{\boldsymbol\theta}\) is Lipschitz and meets
Assumption~\ref{assump:connectivity_functional}.
For any close pair \((\vx,\vy)\) with \(|\vx-\vy|\ll1\) whose straight
segment lies in \(\Omega\), Lemma~\ref{small ggap} and
Proposition~\ref{small path} give
\begin{align}
&\frac{1}{|\vx-\vy|}\left|\int_{0}^1 |\gamma'_{\vx,\vy}(t)|\,g_{\vtheta}\bigl(\gamma_{\vx,\vy}(t)\bigr)\,\D t - \int_{0}^1 |\vx-\vy|\,g_{\vtheta}\Bigl(\vx+t(\vy-\vx)\Bigr)\D t\right|\notag\\\le&\frac{\left|\int_{0}^1 (|\gamma'_{\vx,\vy}(t)|-|\vx-\vy|)g_{\vtheta}\bigl(\gamma_{\vx,\vy}(t)\bigr)\D t \right|}{|\vx-\vy|}+\left|\int_{0}^1 g_{\vtheta}\bigl(\gamma_{\vx,\vy}(t)\bigr)-g_{\vtheta}\Bigl(\vx+t(\vy-\vx)\Bigr)\D t \right|\notag\\=&\fO(|\vx-\vy|).\notag
\end{align}
so we may replace the geodesic by the straight line.
For distant pairs \(|\vx-\vy|\not\ll1\) the weight
\(\eta\bigl(w[\vx,\vy]/\varepsilon\bigr)\) is already negligible, and if the
segment \([\vx,\vy]\) leaves \(\Omega\) we simply discard that pair. Therefore, we can rewrite our optimization problem as
\begin{equation}\label{eq:loss_function1}
\min_{\vtheta} L(\vtheta) := \min_{\vtheta} \frac{1}{|\sA|}\sum_{(i,j)\in \sA} \Biggl|\eta\Biggl(\frac{\frac{|\vx_i-\vx_j|}{N}\sum_{k=0}^N g_{\vtheta}\Bigl(\vx_i+\frac{k}{N}(\vx_j-\vx_i)\Bigr)}{\varepsilon}\Biggr) - w_{ij}\Biggr|^2,
\end{equation}
where \( \sA \subset \{1, \dots, 68\} \times \{1, \dots, 68\} \) denotes the set of parcel pairs considered, \( \vx_i \) and \( \vx_j \) are the centroids of the respective parcels, \( w_{ij} \) is the observed FC between parcels \( i \) and \( j \), \( \varepsilon \) is a positive normalization constant, and \( N \) is the number of quadrature points used along the linear path between \( \vx_i \) and \( \vx_j \).

To simplify training and avoid parameter scaling ambiguities, we absorb the constant \( \varepsilon \) into the network parameters and apply the inverse of \( \eta \) to the data. The modified form of the loss function becomes

\begin{equation}\label{eq:loss_function_modified}
\min_{\vtheta} L(\vtheta) := \min_{\vtheta} \frac{1}{|\sA|}\sum_{(i,j)\in \sA} \Biggl|\Biggl({\frac{|\vx_i-\vx_j|}{N}\sum_{k=0}^N g_{\vtheta}\Bigl(\vx_i+\frac{k}{N}(\vx_j-\vx_i)\Bigr)}\Biggr) - \eta^{-1}(w_{ij})\Biggr|^2. 
\end{equation}

This objective measures the discrepancy between the aggregated predicted diffusivity along parcel connections and the transformed FC data, guiding the training of \( g_\theta \) to align with the observed connectivity structure. In our implementation, we set \( \eta(x) = \exp(-x^2) \). \bluetext{This choice is numerically stable in our setting because we only aggregate functionally connected pairs among spatially nearby parcels, so the resulting FC weights are bounded away from zero and the arguments of $\eta^{-1}$ never approach the unstable regime.}

\subsection{Neural Network Model for Learning Diffusivity}

As described in Section~\ref{Method_Ex}, we construct a data-driven loss function to learn the spatially varying diffusivity $D(\vx)$. To parameterize $D(\vx)$, we employ a neural network model implemented using PyTorch. The network takes a 3D spatial coordinate $\vx \in \Omega \subset \mathbb{R}^3$ as input and outputs a positive scalar value representing the diffusivity at that location. \bluetext{The detailed architecture of the neural network, including activation functions and training procedures, is provided in Supplementary Materials Sec.~\ref{app:nn_details}.}

Using only nearby parcels limited the data size and led to overfitting, while including distant parcels increased data volume but introduced noise and complexity. After exploring both extremes, we found that a carefully balanced dataset yielded the best achievable performance in our setting, though the results still leave some room for improvement.

\begin{remark}
One of the key challenges in training the diffusivity model stemmed from the limited size of the dataset. Although the total number of brain parcels was 68, we only used functionally connected (FC) pairs among spatially nearby parcels, resulting in just \bluetext{the 227 FC edges between parcels.}  This relatively small dataset constrained the model’s capacity to generalize and made robust training difficult. Moreover, because the FC information was associated with the centroids of each parcel, we trained the model using these centroids as representative spatial inputs. This approximation introduced additional errors, as it does not capture intra-parcel variability.
\end{remark}

Using the trained the neural network model, we compute the spatially varying diffusivity $D(\vx)$ as
\[
D(\vx) = \frac{1}{g_\theta(\vx)^{d+2}},
\]
where $g_\theta(\vx)$ is the neural network output and $d = 3$ denotes the spatial dimension. This ensures consistency with the theoretical formulation of the diffusion operator in our model. This $g_\theta(\vx)$ reflects the learned spatial heterogeneity, informed by 
functional connectivity structure.

The initial condition $u_0(\vx)$ for the PDE simulation is extracted from PET scan data, which provides a realistic spatial distribution of the quantity of interest at time $t = 0$. Figure~\ref{fig:final-diffusivity} displays the resulting diffusivity field $D(\vx)$ derived from the trained model and initial condition $u_0(\vx)$ extracted from PET scan data. 

\begin{figure}[ht]
\centering
\begin{overpic}[width=0.3\textwidth]{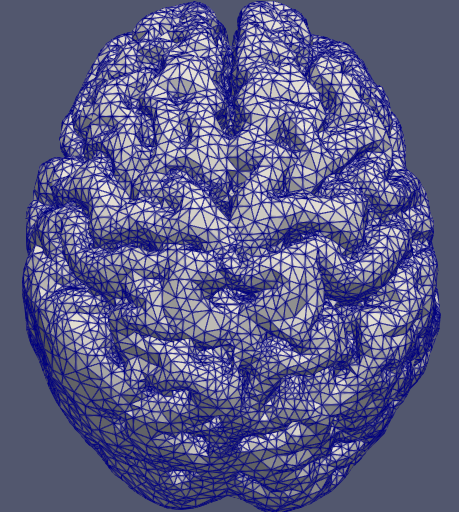}
\end{overpic}
\begin{overpic}[width=0.3\textwidth]{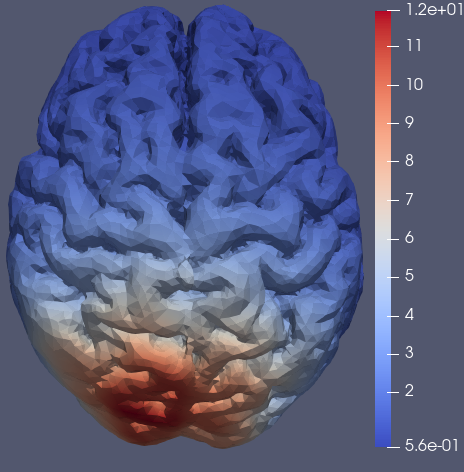}
\put(50,107){\makebox(0,0){\textbf{ $\frac{1}{g_\theta(\vx)^{d+2}}$}}}
\end{overpic}
\begin{overpic}[width=0.29\textwidth]{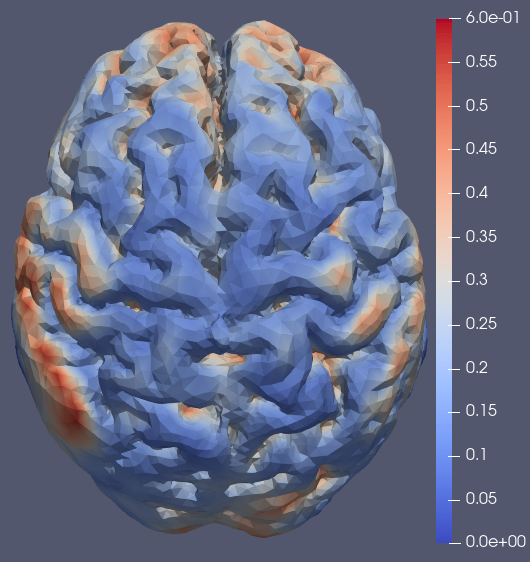}
\put(50,105){\makebox(0,0){\textbf{$ u_0(\vx)$}}}
\end{overpic}
\caption{Left: 3D brain domain reconstructed from MRI data. Middle: Computed diffusivity field $D(\vx)$ obtained from the trained neural network model $D(\vx) = \frac{1}{g_\theta(\vx)^{d+2}}$, with $d = 3$. Right: The initial condition $u_0(\vx)$ was obtained from PET scan imaging.}
\label{fig:final-diffusivity}
\end{figure}

\subsubsection{Recovery of Handcrafted Diffusivity via Neural Network Training}

To evaluate the identifiability of our neural network model and validate the end-to-end training pipeline, we performed a series of inverse experiments. Instead of learning the diffusivity function \( g_\theta(\vx) \) from empirical functional connectivity data, we prescribed a known ground-truth scalar field \( g_{\mathrm{true}}(\vx) \) and used it to compute synthetic weights \( w_{ij} \) according to the integral relation:
\[
    w_{ij} \approx \eta \left( \int_{0}^{1} |\vx_i - \vx_j| \cdot g_{\mathrm{true}}\left( \vx_i + t(\vx_j - \vx_i) \right) \, \D t \right),
\]
where \( \vx_i, \vx_j \in \mathbb{R}^3 \) are points sampled from the brain domain, and \( \eta(x) = \exp(-x^2) \) is a fixed nonlinearity modeling connectivity decay.

\paragraph{Training Procedure}
We sampled \( n \) spatial points \( \{ \vx_1, \dots, \vx_n \} \subset \Omega \) and constructed the training set by collecting all node pairs \( (i,j) \) such that \( |\vx_i - \vx_j| < \varepsilon \), where $\varepsilon$ is chosen according to the scaling in Eq.~\eqref{varepsilon}. For each pair, we evaluated \( w_{ij} \) as above. The neural network model \( g_\theta(\vx) \) was then trained to minimize the following loss function:
\begin{equation} \label{eq:loss_function_modified2}
    \min_{\vtheta} L(\vtheta) := \frac{1}{|\sA|} \sum_{(i,j) \in \sA} \left| \left( \frac{|\vx_i - \vx_j|}{N} \sum_{k=0}^N g_\theta \left( \vx_i + \frac{k}{N}(\vx_j - \vx_i) \right) \right) - \eta^{-1}(w_{ij}) \right|^2,
\end{equation}
where \( \sA \) denotes the set of valid training pairs and \( N \) is the number of interpolation steps used in the discrete approximation of the path integral. \bluetext{This procedure was repeated 100 times with independently sampled training sets, where in each run the network was trained until the loss value decreased below $10^{-7}$ or no further improvement was observed, and the aggregated results across these repetitions were used for comparison.}

\paragraph{Domain Normalization}
Let \( \vx_i^{\mathrm{raw}} \in \mathbb{R}^3 \) denote the original position of node \( i \). Prior to applying any function, we normalize all coordinates using standard score normalization:
\[
    \vx_i = \frac{\vx_i^{\mathrm{raw}} - \boldsymbol{\mu}}{\boldsymbol{\text{Var}}},
\]
where \( \boldsymbol{\mu} \in \mathbb{R}^3 \) and \( \boldsymbol{\text{Var}} \in \mathbb{R}^3 \) are the empirical mean and standard deviation vectors across all points. This normalization centers the domain at the origin and ensures unit variance along each axis.

\paragraph{Ground-Truth Function}
We defined two classes of ground-truth scalar fields on the normalized domain:

\begin{itemize}
    \item \textbf{Sigmoid-based radial field:}
    \[
        g_{\mathrm{true}}(\vx) = \sigma\left( -a ( \|\vx\| - b ) \right) + c,
    \]
    where \( \sigma(t) = 1 / (1 + e^{-t}) \) is the sigmoid function, \( a > 0 \) controls sharpness, \( b \) centers the radial profile, and \( c > 0 \) ensures non-negativity.

    \item\bluetext{\textbf{Cosine-based radial field:}
    \[
        g_{\mathrm{true}}(\vx) = A \cdot \cos\left( \pi - a ( \|\vx\| - b ) \right) + c,
    \]
    where \( A > 0 \) is the amplitude, \( a > 0 \) controls frequency, \( b \) centers the wave radially, and \( c \) is an offset that aligns the function with a prescribed value range.
} 
\end{itemize}

\paragraph{Evaluation Metrics}
To evaluate generalization, we sampled 200 new test points disjoint from training and constructed a validation set by including pairwise distances less than $\varepsilon$, where $\varepsilon$ follows the scaling in Eq.~\eqref{varepsilon}. For direct function recovery, we computed the mean absolute error (MAE) and root mean square error (RMSE):

\begin{equation}
    \label{MAE}
    \text{MAE} = \frac{1}{n} \sum_{i=1}^n |g_{\mathrm{true}}(\vx_i) - g_\theta(\vx_i)|, \quad
    \text{RMSE} = \sqrt{ \frac{1}{n} \sum_{i=1}^n (g_{\mathrm{true}}(\vx_i) - g_\theta(\vx_i))^2 }.
\end{equation}
We also compared the associated diffusivity fields \( D(\vx) = 1 / g(\vx)^5 \), and evaluated the relative errors:
\begin{equation}
\label{RMAE}
    \text{RMAE} = \frac{1}{n} \sum_{i=1}^n \left| \frac{D_{\mathrm{true}}(\vx_i) - D_\theta(\vx_i)}{D_{\mathrm{true}}(\vx_i)} \right|,~
    \text{RRMSE} = \sqrt{ \frac{1}{n} \sum_{i=1}^n \left( \frac{D_{\mathrm{true}}(\vx_i) - D_\theta(\vx_i)}{D_{\mathrm{true}}(\vx_i)} \right)^2 }.
\end{equation}

Tables~\ref{tab:recovery-sigmoid-1d_5}, \ref{tab:recovery-cosine-1d_5} and 
Figs.~\ref{fig:sig_results5}, \ref{fig:cos_results5}, summarize model performance using the ground-truth profiles.

\begin{table}[htbp]
\centering
\footnotesize
\begin{tabular}{|c|c|c|c|c|c|c|}
\hline
\textbf{n} & \textbf{Final Loss} & \textbf{Val Loss} & \textbf{MAE} & \textbf{RMSE} & \textbf{Rel. MAE} & \textbf{Rel. RMSE} \\
\hline
100 & $9.99e-08 $ & $1.84e-04 $ & $3.51e-02 $ & $5.03e-02 $ & $1.91e-01 $ & $2.80e-01 $ \\
200 & $1.00e-07 $ & $6.98e-06 $ & $7.50e-03 $ & $1.14e-02 $ & $4.66e-02 $ & $7.58e-02 $ \\
400 & $1.00e-07 $ & $7.03e-07 $ & $2.86e-03 $ & $4.37e-03 $ & $1.70e-02 $ & $2.60e-02 $ \\
800 & $1.00e-07 $ & $2.43e-07 $ & $1.88e-03 $ & $2.89e-03 $ & $1.10e-02 $ & $1.65e-02 $ \\
1600 & $1.00e-07 $ & $1.69e-07 $ & $1.66e-03 $ & $2.54e-03 $ & $9.68e-03 $ & $1.43e-02 $ \\
\hline
\end{tabular}
\caption{Performance summary for the sigmoid-based ground-truth function with $\varepsilon = C \left({\log n}/{n}\right)^{1/(d+2)}$ chosen according to Eq.~\eqref{varepsilon} (Theorem~\ref{main1}), for varying numbers of parcels $n$. Each value represents the mean over 100 independent experiments to mitigate randomness.}
\label{tab:recovery-sigmoid-1d_5}
\end{table}

\begin{table}[htbp]
\centering
\footnotesize
\begin{tabular}{|c|c|c|c|c|c|c|}
\hline
\textbf{n} & \textbf{Final Loss} & \textbf{Val Loss} & \textbf{MAE} & \textbf{RMSE} & \textbf{Rel. MAE} & \textbf{Rel. RMSE} \\
\hline
100 & $1.32e-07 $ & $1.03e-04 $ & $8.56e-03 $ & $1.59e-02 $ & $6.94e-02 $ & $1.57e-01 $ \\
200 & $2.06e-07 $ & $9.17e-06 $ & $2.33e-03 $ & $4.32e-03 $ & $1.70e-02 $ & $3.09e-02 $ \\
400 & $1.37e-07 $ & $1.27e-06 $ & $1.19e-03 $ & $2.11e-03 $ & $8.39e-03 $ & $1.38e-02 $ \\
800 & $1.83e-07 $ & $7.00e-07 $ & $1.02e-03 $ & $1.65e-03 $ & $7.20e-03 $ & $1.09e-02 $ \\
1600 & $1.20e-07 $ & $3.99e-07 $ & $9.31e-04 $ & $1.47e-03 $ & $6.60e-03 $ & $9.78e-03 $ \\
\hline
\end{tabular}
\caption{Performance summary for the sigmoid-based ground-truth function with $\varepsilon = C \left({\log n}/{n}\right)^{1/(d+2)}$ chosen according to Eq.~\eqref{varepsilon} (Theorem~\ref{main1}), for varying numbers of parcels $n$. Each value represents the mean over 100 independent experiments to mitigate randomness.}
\label{tab:recovery-cosine-1d_5}
\end{table}

\begin{figure}[ht]
    \centering
        \includegraphics[width=0.49\textwidth]{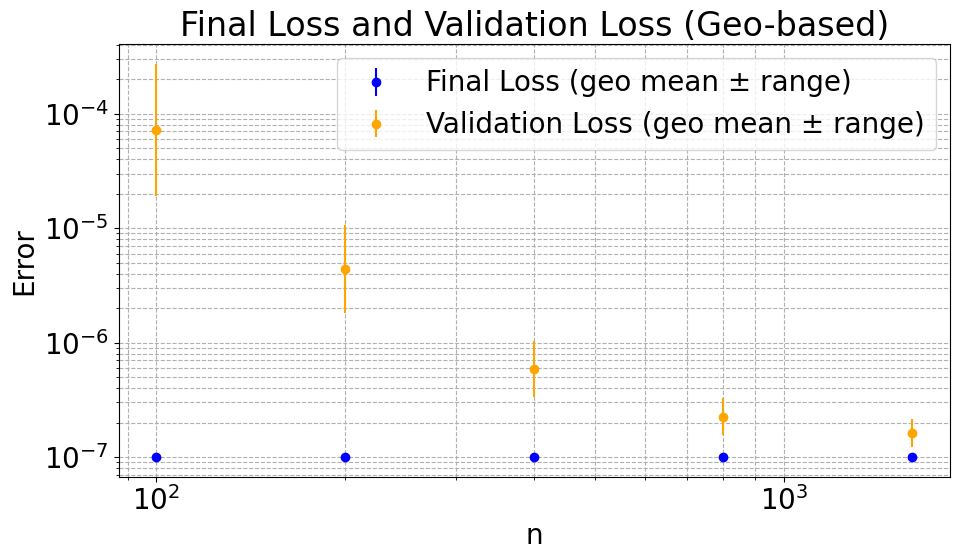}
        \includegraphics[width=0.49\textwidth]{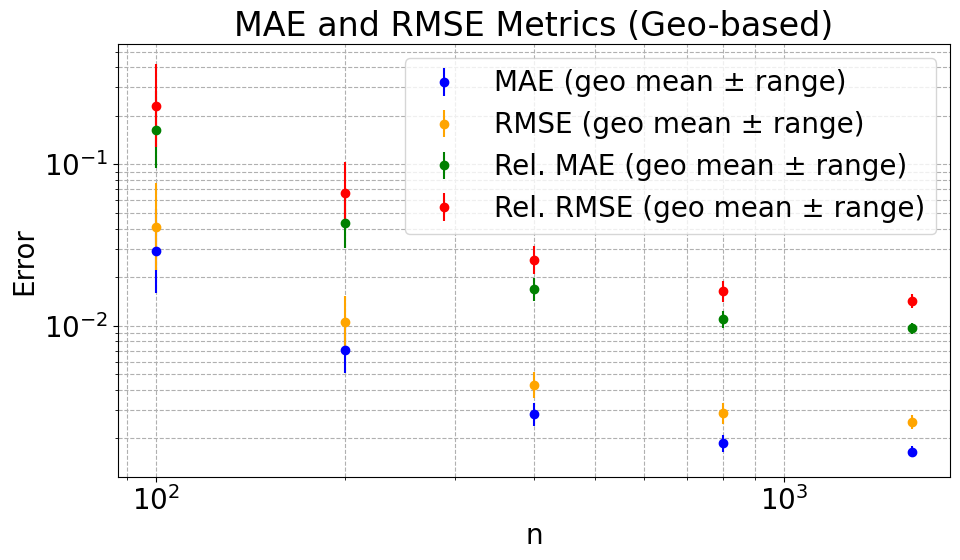}
    \caption{\bluetext{Performance summary for the cosine-based ground-truth function with $\varepsilon = C \left({\log n}/{n}\right)^{1/(d+2)}$ chosen according to Eq.~\eqref{varepsilon} (Theorem~\ref{main1}), for varying numbers of parcels $n$. Each experiment is repeated 100 times with independently sampled datasets. For each metric, we report the geometric mean across these repetitions, together with the multiplicative standard deviation, i.e.\ values of the form $\exp(\mu \pm \sigma)$ where $\mu$ and $\sigma$ denote the mean and standard deviation of the log-transformed results. Both axes are shown in logarithmic scale. For the definitions of Final Loss and Validation Loss, see Eq.~\ref{eq:loss_function_modified2}; for the other metrics, refer to Eqs.~\ref{MAE} and~\ref{RMAE}.}}
    \label{fig:sig_results5}
\end{figure}

\begin{figure}[ht]
    \centering
        \includegraphics[width=0.49\textwidth]{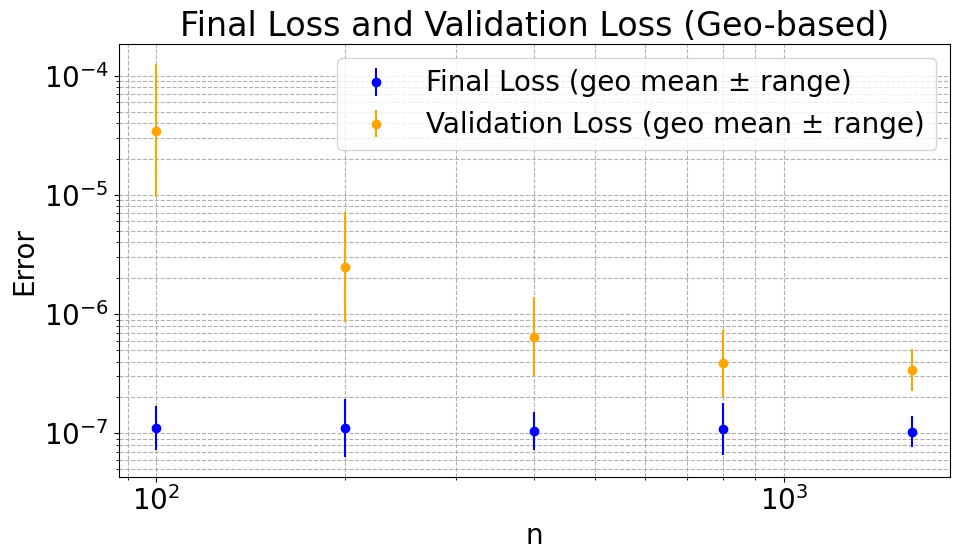}
        \includegraphics[width=0.49\textwidth]{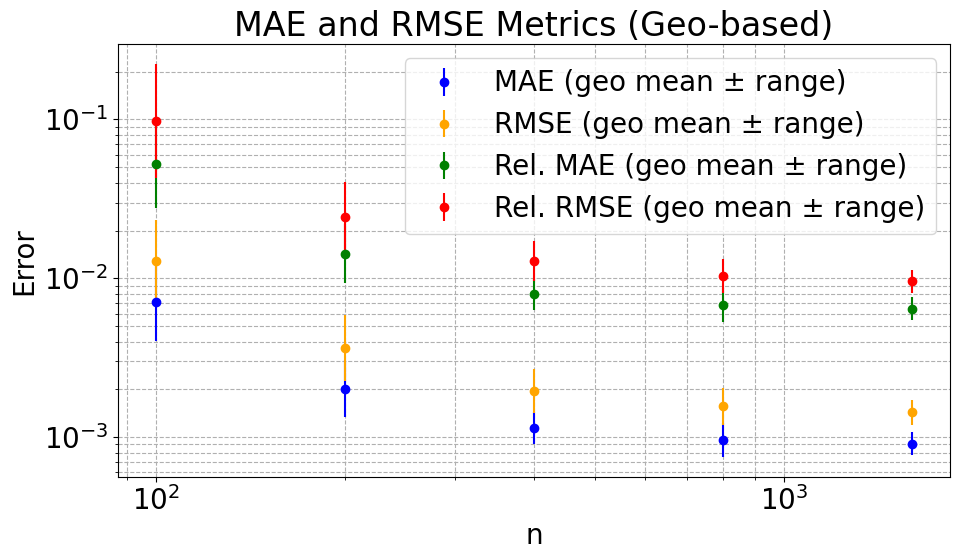}
    \caption{\bluetext{Performance summary for the cosine-based ground-truth function with $\varepsilon = C \left({\log n}/{n}\right)^{1/(d+2)}$ chosen according to Eq.~\eqref{varepsilon} (Theorem~\ref{main1}), for varying numbers of parcels $n$. Each experiment is repeated 100 times with independently sampled datasets. For each metric, we report the geometric mean across these repetitions, together with the multiplicative standard deviation, i.e.\ values of the form $\exp(\mu \pm \sigma)$ where $\mu$ and $\sigma$ denote the mean and standard deviation of the log-transformed results. Both axes are shown in logarithmic scale. For the definitions of Final Loss and Validation Loss, see Eq.~\ref{eq:loss_function_modified2}; for the other metrics, refer to Eqs.~\ref{MAE} and~\ref{RMAE}.}}
    \label{fig:cos_results5}
\end{figure}

Our inverse experiments demonstrate that both sigmoid-based and cosine-based ground-truth diffusivity functions can be accurately recovered by the neural network model, even with relatively few sampled nodes. As shown in Tables~\ref{tab:recovery-sigmoid-1d_5}, \ref{tab:recovery-cosine-1d_5} and Figs.~\ref{fig:sig_results5}, \ref{fig:cos_results5}, increasing the number of training samples beyond a certain threshold yields only marginal improvements in loss and recovery error.

\bluetext{The recovery error decreases almost linearly with the number of samples at first, but the improvement slows down in later stages. Our theoretical results (Theorem~\ref{main1} with the constants in~\ref{varepsilon}) predict a linear decay throughout, so the observed discrepancy requires explanation. We considered several possible causes: finite sample size, numerical error, and training limitations. Increasing the number of samples and refining numerical accuracy (using double precision and denser quadrature points for integration) produced almost no improvement, ruling out the first two. We also tested larger network architectures, but the overall behavior remained the same, with an initial near-linear decrease followed by saturation indicating that the compact model was not the source of the slowdown in the decay. These observations point to training error as the most plausible explanation. Consequently, we continued to use the small architecture adopted in our main experiments, both for consistency and because it facilitates fair comparison across settings. Despite the flattening, the final recovery error remains sufficiently small to be practically useful, showing that even with a relatively small dataset and a simple network architecture, the model achieves reliable and accurate recovery.
}

However, this favorable performance is observed only in the idealized setting where the target function \( g_{\mathrm{true}}(\vx) \) is known exactly and the synthetic weights \( w_{ij} \) are generated without measurement noise. In contrast, when applying the same model and training procedure to real data, we observed noticeably higher final losses and reduced recovery accuracy. This discrepancy can likely be attributed to several sources of error inherent in the real-world setting. First, the observed connectivity values \( w_{ij} \) in empirical data are noisy measurements, often affected by scanner artifacts, preprocessing variability \cite{murphy2013resting}. Second, the spatial positions used in our framework are based on parcellation centroids, which only coarsely approximate the true geometric layout of cortical regions. Finally, our approximation of the line integral via discrete interpolation along straight-line paths introduces additional numerical error, particularly when the underlying diffusivity field is not perfectly aligned with Euclidean geodesics.

Taken together, these factors help explain the performance gap between idealized and empirical recovery. While the model is demonstrably capable of learning smooth, radially structured scalar fields from synthetic data, future work will require more refined representations, noise-aware formulations, and possibly model architectures with greater capacity to address the complexities of real brain connectivity data.

\subsection{Computational Approach to the PDE Model}
To numerically solve the system \eqref{eq:model}, we adopt the finite element method (FEM), a widely used approach for solving partial differential equations (PDEs) on complex geometries such as the brain-shaped domain $\Omega$.

We begin by deriving the weak (variational) formulation of the problem. Let $V := H^1(\Omega)$ be the Sobolev space of square-integrable functions with square-integrable first derivatives. Multiplying the PDE by a test function $v \in V$ and integrating over the domain, we obtain:
\begin{equation}\label{eq:weak_form}
\int_{\Omega} u_t v \, d\vx + \int_{\Omega} D(\vx) \nabla u \cdot \nabla v \, d\vx = \int_{\Omega} C(1 - u) u v \, d\vx,
\end{equation}

We discretize the spatial domain using a conforming finite element space \( V_h \subset H^1(\Omega) \), where \( h > 0 \) denotes the maximum diameter of elements in a shape-regular triangulation \( \mathcal{T}_h \) of \( \Omega \). The space \( V_h \) consists of continuous, piecewise linear functions (i.e., polynomials of degree one) defined over \( \mathcal{T}_h \). Let \( \{ \phi_i \}_{i=1}^{N_h} \) be a basis of \( V_h \), where \( N_h \) is the number of degrees of freedom. The discrete solution \( u^n \in V_h \) at each time step is then sought in this space.

This spatial discretization reduces the original PDE to a finite-dimensional variational problem at each time step. When combined with the semi-implicit time discretization described earlier, it results in a fully discrete scheme that is efficient and stable for solving reaction-diffusion equations.

For time discretization of the weak formulation \eqref{eq:weak_form}, we divide the interval \( [0, T] \) into \( N \) uniform steps of size \( \Delta t \), and denote by \( u^n \in V_h \) the finite element approximation at time \( t_n = n \Delta t \). A semi-implicit time-stepping scheme is employed: the linear diffusion term is treated implicitly, while the nonlinear reaction term is evaluated explicitly at the previous time step to avoid solving a fully nonlinear system.

More precisely, for each \( n \ge 0 \), we seek \( u^{n+1} \in V_h \) such that for all \( v \in V_h \),
\[
\int_{\Omega} u^{n+1} v \, d\vx + \Delta t \int_{\Omega} D(\vx) \nabla u^{n+1} \cdot \nabla v \, d\vx
= \int_{\Omega} u^n v \, d\vx + \Delta t \int_{\Omega} C u^n (1 - u^n) v \, d\vx.
\]

This semi-implicit formulation improves computational efficiency by linearizing the nonlinear reaction term, while maintaining stability through the implicit treatment of the diffusion term.

Using the previously obtained diffusion coefficient \( D(x) \) and the initial condition \( u_0 \), 
we solve the PDE \eqref{eq:model} and analyze the effect of spatially varying diffusion. Figure~\ref{fig:results0} shows the time evolution of the integral difference over the boundary region between the solutions obtained using the learned \( D(x) \) and a constant diffusion coefficient \( \overline{D} = \min (D(x)) \), under the case \( C = 0 \) and \( C = 0.02 \). The quantity plotted is:
\[
\int_{\partial \Omega} \delta u(x,t) \, dS = \int_{\partial \Omega} u_{\overline{D}}(x,t)-u_{D(x)}(x,t) \, dS.
\]

We chose to compute the integral difference only over the boundary region because the majority of the training data used to learn the diffusion coefficient \( D(x) \) is concentrated near the boundary of the domain. As we move toward the interior of the domain, the available information becomes increasingly sparse, making the learned \( D(x) \) less reliable in those regions. Therefore, to ensure a meaningful and fair comparison between the solutions computed using the learned and constant diffusion coefficients, we restrict the evaluation of the integral to the boundary.

Furthermore, Figure~\ref{fig:results0} visualizes the brain image at the time point when this difference \( \int_{\partial \Omega} \delta u(x,t) \, dS \) reaches its maximum, highlighting spatial effects due to the heterogeneity in \( D(x) \).

\begin{figure}
    \centering
    \begin{overpic}[width=0.39\linewidth]{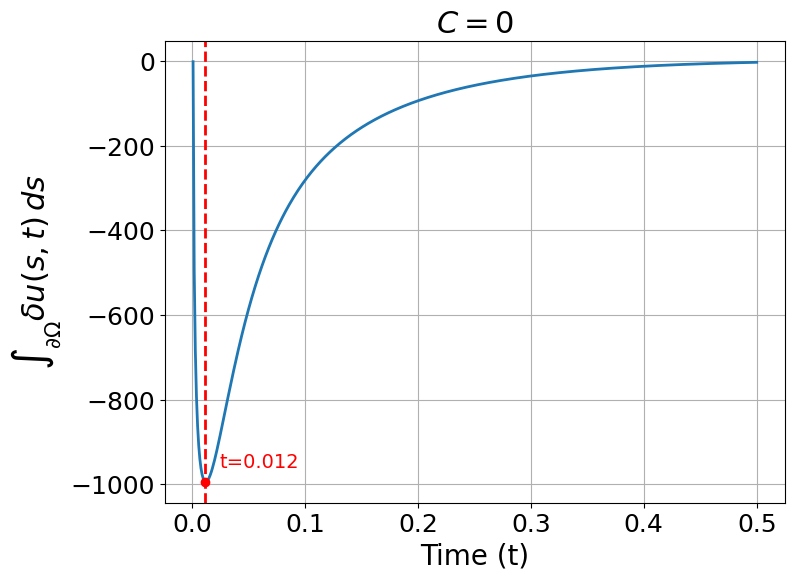}
    \end{overpic}
    \begin{overpic}[width=0.29\linewidth]{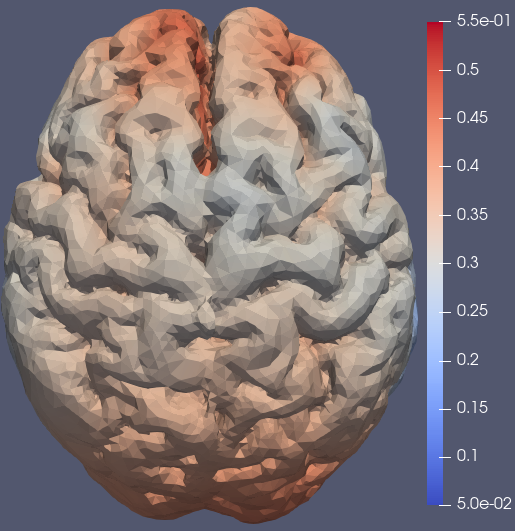}
    \put(50,106){\makebox(0,0){\textbf{\small $t=0.012, D(\vx) = \frac{1}{g_\theta(\vx)^{d+2}}$}}}
    \end{overpic}
    \begin{overpic}[width=0.29\linewidth]{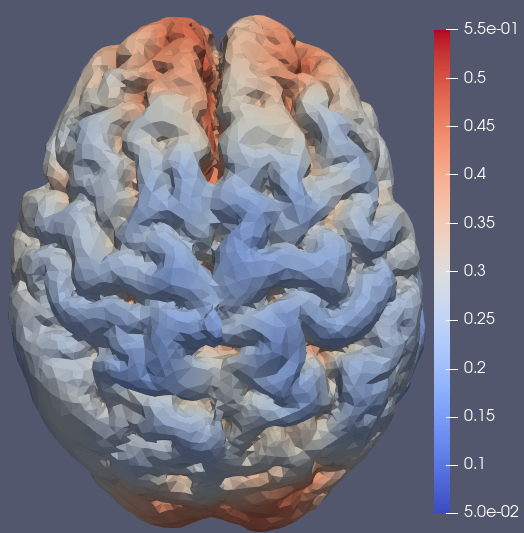}
    \put(50,107){\makebox(0,0){\textbf{\small $\overline{D} = \min\left(\frac{1}{g_\theta(\vx)^{d+2}}\right)$}}}
    \end{overpic}
\par\vspace{0.8cm}
    \begin{overpic}[width=0.38\linewidth]{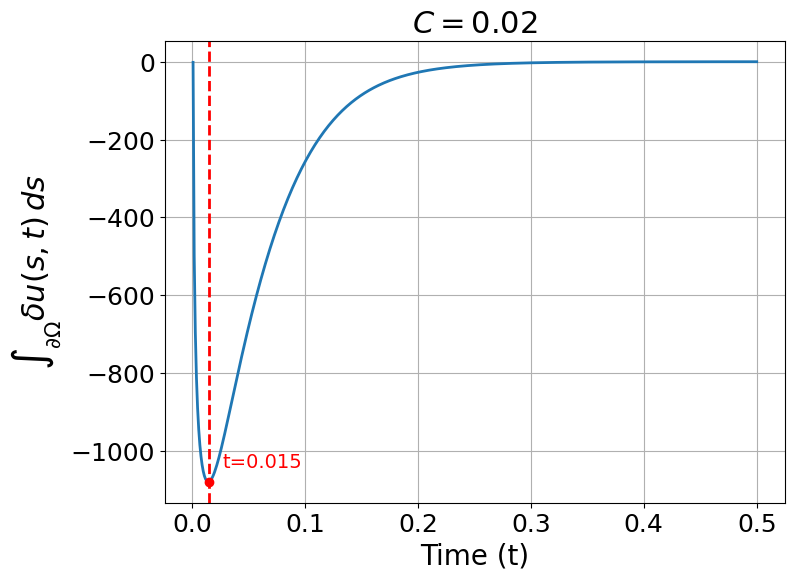}
    \end{overpic}
    \begin{overpic}[width=0.29\linewidth]{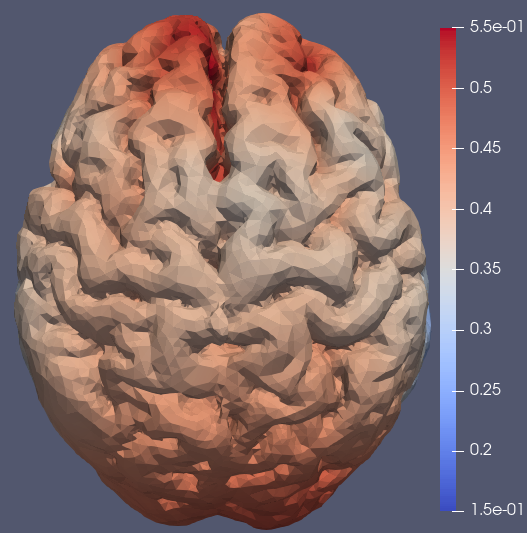}
    \put(50,106){\makebox(0,0){\textbf{\small $t=0.015, D(\vx) = \frac{1}{g_\theta(\vx)^{d+2}}$}}}
    \end{overpic}
    \begin{overpic}[width=0.29\linewidth]{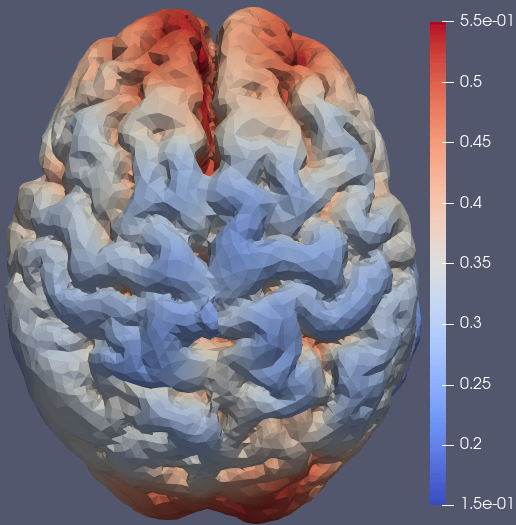}
    \put(50,107){\makebox(0,0){\textbf{\small $\overline{D} = \min\left(\frac{1}{g_\theta(\vx)^{d+2}}\right)$}}}
    \end{overpic}
    \caption{Left: Time evolution of the difference in the boundary integral between the solutions computed using the trained diffusion coefficient \( D(x) \) and a constant diffusion coefficient. 
Middle: $u_{D(x)}$ at the time point when this difference is maximized, using the trained \( D(x) \).
Right: $u_{\overline{D}}$ at the same time point using the constant diffusion coefficient.
The constant value used for \( D(x) \) corresponds to the minimum value of the trained diffusion coefficient, i.e., \( \overline{D} = \min(D(x)) \).}
    \label{fig:results0}
\end{figure}

\section{Conclusion}

In this paper, we derived a continuum model for the $\varepsilon$-graph with a general connectivity functional via an energy-based approach. We proved that the energies of the discrete $\varepsilon$-graph and its continuum limit agree up to an error of order $\mathcal{O}(\varepsilon)$. Importantly, the constant in this estimate depends only on the $W^{1,1}$-norm of the connectivity density, so our result remains valid even when the density exhibits large local variations. We then applied the continuum model to brain dynamics and showed that, by introducing a spatially varying diffusion coefficient rather than a constant one, our model more accurately captures connectivity-driven effects than the classical formulation at the onset of the dynamics.

Looking ahead, we identify two main directions for future research. First, we aim to strengthen the theoretical framework by establishing convergence of the minimizers themselves—demonstrating that the discrete and continuum solutions are close, rather than only their associated energies. Proving such stability estimates is a challenging problem that will likely require new analytical techniques. Second, we plan to extend our methodology to systems with nonlocal interactions or anisotropic structures. In these cases, the continuum limit is expected to involve nonlocal PDEs or higher-order analysis incorporating directional information, and determining their precise form and properties remains an important open challenge.

\bibliographystyle{siamplain}
\bibliography{references}

\begin{thebibliography}{10}

\bibitem{arcones1995bernstein}
{\sc M.~A. Arcones}, {\em A {B}ernstein-type inequality for {U}-statistics and {U}-processes}, Statistics \& probability letters, 22 (1995), pp.~239--247.

\bibitem{bardi1997optimal}
{\sc M.~Bardi, I.~C. Dolcetta, et~al.}, {\em Optimal control and viscosity solutions of Hamilton-Jacobi-Bellman equations}, vol.~12, Springer, 1997.

\bibitem{bungert2024convergence}
{\sc L.~Bungert, J.~Calder, M.~Mihailescu, K.~Houssou, and A.~Yuan}, {\em Convergence rates for poisson learning to a poisson equation with measure data}, arXiv preprint arXiv:2407.06783,  (2024).

\bibitem{bungert2023uniform}
{\sc L.~Bungert, J.~Calder, and T.~Roith}, {\em Uniform convergence rates for lipschitz learning on graphs}, IMA Journal of Numerical Analysis, 43 (2023), pp.~2445--2495.

\bibitem{burago2001course}
{\sc D.~Burago, Y.~Burago, S.~Ivanov, et~al.}, {\em A course in metric geometry}, vol.~33, American Mathematical Society Providence, 2001.

\bibitem{calder2018game}
{\sc J.~Calder}, {\em The game theoretic p-laplacian and semi-supervised learning with few labels}, Nonlinearity, 32 (2018), p.~301.

\bibitem{calder2020calculus}
{\sc J.~Calder}, {\em The calculus of variations}, University of Minnesota, 40 (2020).

\bibitem{calder2022lipschitz}
{\sc J.~Calder, N.~Garc{\'i}a~Trillos, and M.~Lewicka}, {\em Lipschitz regularity of graph laplacians on random data clouds}, SIAM Journal on Mathematical Analysis, 54 (2022), pp.~1169--1222.

\bibitem{calder2022improved}
{\sc J.~Calder and N.~G. Trillos}, {\em Improved spectral convergence rates for graph laplacians on $\varepsilon$-graphs and k-nn graphs}, Applied and Computational Harmonic Analysis, 60 (2022), pp.~123--175.

\bibitem{craddock2012whole}
{\sc R.~C. Craddock, G.~A. James, P.~E. Holtzheimer~III, X.~P. Hu, and H.~S. Mayberg}, {\em A whole brain fmri atlas generated via spatially constrained spectral clustering}, Human brain mapping, 33 (2012), pp.~1914--1928.

\bibitem{crank1979mathematics}
{\sc J.~Crank}, {\em The mathematics of diffusion}, Oxford university press, 1979.

\bibitem{garcia2020error}
{\sc N.~Garc{\'\i}a~Trillos, M.~Gerlach, M.~Hein, and D.~Slep{\v{c}}ev}, {\em Error estimates for spectral convergence of the graph laplacian on random geometric graphs toward the laplace--beltrami operator}, Foundations of Computational Mathematics, 20 (2020), pp.~827--887.

\bibitem{hao2016mathematical}
{\sc W.~Hao and A.~Friedman}, {\em Mathematical model on alzheimer’s disease}, BMC systems biology, 10 (2016), p.~108.

\bibitem{hao2025optimal}
{\sc W.~Hao, C.-Y. Kao, S.~Lee, and Z.~Li}, {\em Optimal control for anti-abeta treatment in alzheimer's disease using a reaction-diffusion model}, arXiv preprint arXiv:2504.07913,  (2025).

\bibitem{hao2022optimal}
{\sc W.~Hao, S.~Lenhart, and J.~R. Petrella}, {\em Optimal anti-amyloid-beta therapy for alzheimer’s disease via a personalized mathematical model}, PLoS computational biology, 18 (2022), p.~e1010481.

\bibitem{liu2007ab}
{\sc F.~Liu, P.~Ming, and J.~Li}, {\em Ab initio calculation of ideal strength and phonon instability of graphene under tension}, Physical Review B—Condensed Matter and Materials Physics, 76 (2007), p.~064120.

\bibitem{luo2018atomistic}
{\sc T.~Luo, P.~Ming, and Y.~Xiang}, {\em From atomistic model to the peierls--nabarro model with $\gamma$-surface for dislocations}, Archive for Rational Mechanics and Analysis, 230 (2018), pp.~735--781.

\bibitem{luskin2013atomistic}
{\sc M.~Luskin and C.~Ortner}, {\em Atomistic-to-continuum coupling}, Acta Numerica, 22 (2013), pp.~397--508.

\bibitem{makridakis2011priori}
{\sc C.~Makridakis, C.~Ortner, and E.~S{\"u}li}, {\em A priori error analysis of two force-based atomistic/continuum models of a periodic chain}, Numerische Mathematik, 119 (2011), pp.~83--121.

\bibitem{ming2007cauchy}
{\sc P.~Ming and W.~E}, {\em Cauchy--born rule and the stability of crystalline solids: static problems}, Archive for rational mechanics and analysis, 183 (2007), pp.~241--297.

\bibitem{murphy2013resting}
{\sc K.~Murphy, R.~M. Birn, and P.~A. Bandettini}, {\em Resting-state fmri confounds and cleanup}, Neuroimage, 80 (2013), pp.~349--359.

\bibitem{murray2007mathematical}
{\sc J.~D. Murray}, {\em Mathematical biology: I. An introduction}, vol.~17, Springer Science \& Business Media, 2007.

\bibitem{nakatsukasa2018aaa}
{\sc Y.~Nakatsukasa, O.~S{\`e}te, and L.~N. Trefethen}, {\em The aaa algorithm for rational approximation}, SIAM Journal on Scientific Computing, 40 (2018), pp.~A1494--A1522.

\bibitem{osada2006homogenization}
{\sc H.~Osada}, {\em Homogenization of diffusion processes with random stationary coefficients}, in Probability Theory and Mathematical Statistics: Proceedings of the Fourth USSR-Japan Symposium, held at Tbilisi, USSR, August 23--29, 1982, Springer, 2006, pp.~507--517.

\bibitem{raj2012network}
{\sc A.~Raj, A.~Kuceyeski, and M.~Weiner}, {\em A network diffusion model of disease progression in dementia}, Neuron, 73 (2012), pp.~1204--1215.

\bibitem{schaefer2018local}
{\sc A.~Schaefer, R.~Kong, E.~M. Gordon, T.~O. Laumann, X.-N. Zuo, A.~J. Holmes, S.~B. Eickhoff, and B.~T. Yeo}, {\em Local-global parcellation of the human cerebral cortex from intrinsic functional connectivity mri}, Cerebral cortex, 28 (2018), pp.~3095--3114.

\bibitem{shahhosseini2022functional}
{\sc Y.~Shahhosseini and M.~F. Miranda}, {\em Functional connectivity methods and their applications in fmri data}, Entropy, 24 (2022), p.~390.

\bibitem{slepcev2019analysis}
{\sc D.~Slepcev and M.~Thorpe}, {\em Analysis of p-laplacian regularization in semisupervised learning}, SIAM Journal on Mathematical Analysis, 51 (2019), pp.~2085--2120.

\bibitem{sporns2011human}
{\sc O.~Sporns}, {\em The human connectome: a complex network}, Annals of the new York Academy of Sciences, 1224 (2011), pp.~109--125.

\bibitem{trillos2025minimax}
{\sc N.~G. Trillos, C.~Li, and R.~Venkatraman}, {\em Minimax rates for the estimation of eigenpairs of weighted laplace-beltrami operators on manifolds}, arXiv preprint arXiv:2506.00171,  (2025).

\bibitem{yang2024near}
{\sc Y.~Yang and Y.~Lu}, {\em Near-optimal deep neural network approximation for {Korobov} functions with respect to {$L_p$ and $H_1$} norms}, Neural Networks, 180 (2024), p.~106702.

\bibitem{yang2022convergence}
{\sc Y.~Yang, T.~Luo, and Y.~Xiang}, {\em Convergence from atomistic model to peierls--nabarro model for dislocations in bilayer system with complex lattice}, Communications in Mathematical Sciences, 20 (2022), pp.~947--986.

\bibitem{yang2023nearly1}
{\sc Y.~Yang, Y.~Wu, H.~Yang, and Y.~Xiang}, {\em Nearly optimal approximation rates for deep super {ReLU} networks on {Sobolev spaces}}, arXiv preprint arXiv:2310.10766,  (2023).

\bibitem{yang2023nearly}
{\sc Y.~Yang, H.~Yang, and Y.~Xiang}, {\em Nearly optimal {VC-dimension} and pseudo-dimension bounds for deep neural network derivatives}, Advances in Neural Information Processing Systems, 36 (2023), pp.~21721--21756.

\bibitem{yang2023stochastic}
{\sc Y.~Yang, L.~Zhang, and Y.~Xiang}, {\em Stochastic continuum models for high-entropy alloys with short-range order}, Multiscale Modeling \& Simulation, 21 (2023), pp.~1323--1343.

\bibitem{zheng2022data}
{\sc H.~Zheng, J.~R. Petrella, P.~M. Doraiswamy, G.~Lin, W.~Hao, and A.~D.~N. Initiative}, {\em Data-driven causal model discovery and personalized prediction in alzheimer's disease}, NPJ digital medicine, 5 (2022), p.~137.

\end{thebibliography}
\appendix
\section{Proof of Proposition \ref{small path}}

{\blue 

\begin{lemma}[\cite{bungert2023uniform}, Proposition 5.1]\label{geodesic distance}
  Suppose that $\Omega$ has a $C^{1,1}$ boundary, then for the geodesic distance \[
d_\Omega(\vx,\vy)=\inf_{\substack{\gamma(0)=\vx, \\ \gamma(1)=\vy, \\ \gamma \in C^1([0,1];\Omega)}}
\int_0^1 |\gamma'(t)|\,\D  t,
\] we have \begin{equation}
d_\Omega(\vx,\vy) \leq |\vx-\vy| + B|\vx-\vy|^{2},\notag
\end{equation}
for $|\vx-\vy|\leq r_\Omega$, where, $r_\Omega$ depends only on the domain, and $B$ is a constant independent of $\vx,\vy$. 
\end{lemma}
By Assumption~\ref{assump:connectivity_functional} we have, for all $\vx,\vy\in\Omega$,
\[
  \underline{g}\,d_{\Omega}(\vx,\vy)\ \le\ d_{g}(\vx,\vy)\ \le\ \overline{g}\,d_{\Omega}(\vx,\vy).
\]
Thus $d_g$ and $d_\Omega$ are bi-Lipschitz equivalent; in particular, if $(\Omega,d_\Omega)$ is complete (e.g., for bounded Lipschitz domains with the intrinsic metric), then $(\Omega,d_g)$ is also complete. By the Hopf–Rinow theorem for length spaces \cite{burago2001course}, for any $\vx,\vy\in\Omega$ there exists a minimizing Lipschitz curve $\gamma:[0,1]\to\Omega$ joining $\vx$ to $\vy$ satisfies
\[
   \int_0^1 g\bigl(\gamma(t)\bigr)\,|\gamma'(t)|\,dt
  \;=\; d_g(\vx,\vy).
\]
In what follows we use this fact and do not further discuss existence issues for minimizing paths.}




\begin{lemma}\label{two bound}
For any two distinct points {\blue \(\vx,\vy\in B_{r_\Omega}(\vx)\cap \Omega\)}, if the radius \(r\) is chosen sufficiently large so that 
\begin{equation}\label{eq:ratio_condition}
\frac{\bar{g}}{\underline{g}} < \frac{2r}{|\vx-\vy|+B|\vx-\vy|^2}-1,
\end{equation}
then any optimal weighted path \(\gamma_{\vx,\vy}:[0,1]\to\mathbb{R}^n\) connecting \(\vx\) and \(\vy\) (i.e., with \(d(\vx,\vy)=\int_0^1g(\gamma_{\vx,\vy}(t))|\gamma_{\vx,\vy}(t)|\D t\)) remains entirely within \(B_r(\vx)\), {\blue where $B$ and $r_\Omega$ are defined in Lemma \ref{geodesic distance}.}
\end{lemma}

\begin{proof}
Suppose by way of contradiction that an optimal path $\gamma_{\vx,\vy}:[0,1]\to\mathbb{R}^n$, with $\gamma_{\vx,\vy}(0)=\vx$ and $\gamma_{\vx,\vy}(1)=\vy$, does \emph{not} remain entirely within $B_r(\vx)$. Define
\[
t_* := \inf\{t\in [0,1] : \gamma_{\vx,\vy}(t) \notin B_r(\vx)\}
\]
and
\[
s_* := \sup\{t\in [0,1] : \gamma_{\vx,\vy}(t) \notin B_r(\vx)\}.
\]
Then the portions of $\gamma_{\vx,\vy}$ on the intervals $[0,t_*]$ and $[s_*,1]$ lie completely within $B_r(\vx)$.

Since $\gamma_{\vx,\vy}(0)=\vx$ and $\gamma_{\vx,\vy}(t_*)\in \partial B_r(\vx)$, the Euclidean length of $\gamma_{\vx,\vy}|_{[0,t_*]}$ satisfies
\[
\int_0^{t_*} |\gamma_{\vx,\vy}'(t)|\,\D  t \ge r.
\]
Similarly, noting that $\vy\in B_r(\vx)$ based on \[1\le\frac{\bar{g}}{\underline{g}} < \frac{2r}{|\vx-\vy|+B|\vx-\vy|^2}-1\Rightarrow r>|\vx-\vy|,\] and $\gamma_{\vx,\vy}(s_*)\in\partial B_r(\vx)$, we have
\[
\int_{s_*}^1 |\gamma_{\vx,\vy}'(t)|\,\D  t \ge r-|\vx-\vy|.
\]
Thus, the total Euclidean length of these two segments is at least
\[
r + \bigl(r-|\vx-\vy|\bigr)=2r-|\vx-\vy|.
\]

Since $g(\vz)\ge \underline{g}$ for all $\vz\in B_r(\vx)\cap\Omega\subset \Omega$, the weighted length along these segments is bounded below by
\[
{\blue d_g(\vx,\vy)} \ge \underline{g}\left(\int_0^{t_*}|\gamma_{\vx,\vy}'(t)|\,\D  t + \int_{s_*}^1 |\gamma_{\vx,\vy}'(t)|\,\D  t\right) \ge \underline{g}(2r-|\vx-\vy|).
\]
On the other hand, {\blue based on $|\vx-\vy|\le r_\Omega$ and Lemma \ref{geodesic distance}, we have}
\[
d_g(\vx,\vy) \le \bar{g}\cdot d_\Omega(\vx,\vy)\le \bar{g}(|\vx-\vy|+B|\vx-\vy|^2).
\]
Thus,
\[
\underline{g}(2r-|\vx-\vy|) \le \bar{g}(|\vx-\vy|+B|\vx-\vy|^2),
\]
or equivalently,
\[
\frac{\bar{g}}{\underline{g}} \ge \frac{2r-|\vx-\vy|}{|\vx-\vy|+B|\vx-\vy|^2}\ge \frac{2r}{|\vx-\vy|+B|\vx-\vy|^2}-1,
\]
which contradicts the assumption \eqref{eq:ratio_condition}.
\end{proof}

\begin{lemma}\label{small ggap}
{\blue Suppose that $\Omega$ has a $C^{1,1}$ boundary, and $g\in W^{1,\infty}(\Omega)$,} and Assumption \ref{assump:connectivity_functional} holds. For any \(r>0\) with $ r< \min\{r_\Omega/2, 1/ B\}$, 
we have
\begin{align}\label{eq:small_path2}
\sup_{\vy,\vz\in \Omega\cap B_r(\vx)} |g(\vz)-g(\vy)| \le 6r\sup_{\va\in\Omega\cap B_{4 r}(\vx)}|\nabla g(\va)| \notag
\end{align}
for any $\vx\in\Omega$, where $B$ and $r_\Omega$ are defined in Lemma \ref{geodesic distance}.
\end{lemma}

\begin{proof}
    For any two points \(\vy,\vz\in \Omega\cap B_{r}(\vx)\), let
\(\gamma_{\Omega,\vz,\vy}\) be a geodesic
minimiser joining \(\vz\) to \(\vy\) to make \[ d_{\Omega}(\vz,\vy)=\int_{0}^{1}|\gamma_{\Omega,\vz,\vy}'(t)|\,\D t.\]
Since \(g\in W^{1,\infty}(\Omega)\) is Lipschitz,
the composition \(t\mapsto g\bigl(\gamma_{\Omega,\vz,\vy}(t)\bigr)\) is
Lipschitz on \([0,1]\) and hence {\blue we have}
\[
  g(\vz)-g(\vy)
  \;=\;\int_{0}^{1}
     (g\circ \gamma_{\Omega,\vz,\vy})'(t)
     \,\D t=\;
  \int_{0}^{1}
     \nabla g\bigl(\gamma_{\Omega,\vz,\vy}(t)\bigr)\cdot
     \gamma_{\Omega,\vz,\vy}'(t)
     \,\D t.
\] Based on $|\vz-\vy|\le 2r\le r_\Omega$, we have that 
\[\int_0^1 |\gamma_{\Omega,\vz,\vy}'(t)|\D t\le |\vz-\vy|+B|\vz-\vy|^2\le 2r+4Br^2\le 6r\]
due to Lemma \ref{geodesic distance} and $r\le 1/B$. 
Furthermore, we also know that $\gamma_{\Omega,\vz,\vy}(t)$ remains entirely within \(B_{4 r}(\vx)\) due to the fact that \[\underbrace{\frac12\int_{0}^{1}\!\bigl|\gamma'_{\Omega,\vz,\vy}(t)\bigr|\,dt}_{\text{maximum distance the path can exit }B_{r}(\vx)}
+r\le 4r.\] Then for any two points $\vz,\vy$ in $B_r(\vx)$, we have that
\begin{equation}
    g(\vz)-g(\vy)\le\sup_{\va\in\Omega\cap B_{4 r}(\vx)}|\nabla g(\va)|\cdot \int_0^1 |\gamma_{\Omega,\vz,\vy}'(t)|\D t\le 6r\sup_{\va\in\Omega\cap B_{4 r}(\vx)}|\nabla g(\va)|,\label{important}
\end{equation}
taking the supremum yields the desired bound.  The term
\[
     \sup_{\va\in\Omega\cap B_{4 r}(\vx)}|\nabla g(\va)|
\]
is finite due to $g\in W^{1,\infty}(\Omega)$.
\end{proof}

Now we can finish the proof of Proposition \ref{small path}.
\begin{proof}[Proof of Proposition \ref{small path}]
For any $\varepsilon>0$, we set \(r=\lambda\eps=\Bigl(\frac{\bar{g}}{\lgg}+1\Bigr)\varepsilon.\) By Lemmas~\ref{geodesic distance} and \ref{two bound}, for any \(\vx,\vy\in\Omega\) with $|\vx-\vy|\le \varepsilon$ we have
\begin{equation}\label{eq:distance_bounds}
\underline{g}_r(\vx)\,|\vx-\vy| \le d_g(\vx,\vy) \le \bar{g}_r(\vx)(|\vx-\vy|+B|\vx-\vy|^2),
\end{equation}
where
\[
\bar{g}_r(\vx) := \sup_{y\in B_r(\vx)} g(y)
\quad \text{and} \quad
\underline{g}_r(\vx) := \inf_{y\in B_r(\vx)} g(y).
\]
This inequality is valid due to 
\[
r=\Bigl(\frac{\bar{g}}{\lgg}+1\Bigr)\varepsilon \Rightarrow \frac{\bar{g}}{\lgg} <\frac{2r}{|\vx-\vy|+B|\vx-\vy|^2}-1,
\]for $|\vx-\vy|\le \varepsilon<1/B$.

Now we show an absolute bound for $d_g(\vx,\vy)-g(\vx)|\vx-\vy|$ when $d(\vx,\vy)\le \varepsilon$. The upper bound can be found by \begin{align}
    d_g(\vx,\vy)-g(\vx)|\vx-\vy|\le (\bar{g}_r(\vx)-g(\vx))\,|\vx-\vy|+B\bar{g}_r(\vx)|\vx-\vy|^2.\notag
\end{align} Due to $|\vx-\vy|\le \frac{d_g(\vx,\vy)}{\lgg}\le \frac{\varepsilon}{\lgg}$, we have \begin{align}
    d_g(\vx,\vy)-g(\vx)|\vx-\vy|\le (\bar{g}_r(\vx)-g(\vx))\,\frac{\varepsilon}{\lgg}+B\bar{g}\frac{\varepsilon^2}{\lgg^2}.\notag
 \end{align} 
Furthermore, based on Lemma \ref{small ggap}, one can estimate
\[
  \bigl|\bar{g}_r(\vx) - g(\vx)\bigr|
  \;\le\; 6r \sup_{\va\in \Omega \cap B_{4\lambda\varepsilon}(\vx)} |\nabla g(\va)|.
\]
Substituting $r = \lambda\varepsilon$ leads to
\[
  \bigl|\bar{g}_r(\vx) - g(\vx)\bigr|
  \;\le\; 6\lambda\varepsilon \sup_{\va\in \Omega \cap B_{4\lambda\varepsilon}(\vx)} |\nabla g(\va)|.
\]
Therefore, the upper bound of $d_g(\vx,\vy) - g(\vx)\,|\vx-\vy|$ is
\[
  \frac{\varepsilon^{2}}{\lgg^{2}}
  \left[\,6\lambda \sup_{\va\in \Omega \cap B_{4\lambda\varepsilon/\lgg}(\vx)} |\nabla g(\va)|
  \;+\; B\bar g \right].
\]

Similarly, for the lower bound, we have
\[
  d_g(\vx,\vy) - g(\vx)\,|\vx-\vy|
  \;\ge\; \bigl(\lgg_r(\vx) - g(\vx)\bigr)|\vx-\vy|
  \;\ge\; -\frac{6\lambda \varepsilon^{2}}{\lgg^{2}}
     \sup_{\va\in \Omega \cap B_{4\lambda\varepsilon/\lgg}(\vx)} |\nabla g(\va)|.
\]
Thus we obtain the desired two–sided bounds.
\end{proof}

\section{Proof of Theorem \ref{main1}}
\begin{lemma}[Bernstein for U-statistics \cite{arcones1995bernstein}]\label{lem:bernstein}
Let \(X_1, \ldots, X_n\) be i.i.d.\ random variables taking values in \(\fX\) and let \(f:\fX^2 \rightarrow \mathbb{R}\) be a symmetric function (i.e., \(f(x,y)=f(y,x)\)). Define
\[
\mu = \mathbb{E}\bigl[f(X_i, X_j)\bigr], \quad \sigma^2 = \operatorname{Var}\bigl(f(X_i, X_j)\bigr) = \mathbb{E}\Bigl[\Bigl(f(X_i, X_j)-\mu\Bigr)^2\Bigr],
\]
and let
\[
b = \|f\|_{\infty}.
\]
Define the U-statistic
\begin{equation}\label{eq:U_statistic}
    U_n = \frac{1}{n(n-1)} \sum_{i \neq j} f\bigl(X_i, X_j\bigr).
\end{equation}
Then, for every \(t>0\),
\begin{equation}\label{eq:bernstein}
    \mathbb{P}\Bigl(U_n - \mu \geq t\Bigr) \leq \exp\left(-\frac{n t^2}{6\left(\sigma^2 + \frac{1}{3} b t\right)}\right).
\end{equation}
\end{lemma}

This inequality is key in showing that the discrete energy \(E_n[u]\) converges to the continuum nonlocal energy \(I_\varepsilon[u]\) as \(n \to \infty\) (with appropriate scaling), thereby linking the discrete graph-based formulations to the continuum energy.

\begin{proof}[Proof of Theorem \ref{main1}]
    Define the U-statistic
\[
U_n=\frac{1}{n(n-1)} \sum_{i \neq j} f\left(\vx_i, \vx_j\right),
\]
where
\[
f(\vx, \vy)=\eta_{\varepsilon}(|\vx-\vy|)\left(\frac{u(\vx)-u(\vy)}{\varepsilon}\right)^2,
\]
and note that $E_n[u]=\frac{n-1}{\sigma_{\eta}n} U_n$, where \begin{equation}
   \eta_{\varepsilon}(|\vx-\vy|):=\frac{\eta\left( \frac{d_g(\vx,\vy)}{\eps}\right)}{\varepsilon^d}.
\end{equation} Due to that $u$ is Lipschitz, $ \eta\le 1$, and \begin{align}
    f(\vx, \vy)\le& \text{Lip}^2(u)|\vx-\vy|^2\varepsilon^{-2-d}\eta\left(\frac{d_\varepsilon(\vx,\vy)}{\varepsilon}\right)\le \text{Lip}^2(u)|\vx-\vy|^2\varepsilon^{-2-d}
\end{align} Otherwise is zero. Now, let us consider the region of $d_g(\vx,\vy)\le \varepsilon$, based on Assumption \ref{assump:connectivity_functional}, we have that \begin{equation}
    \{(\vx,\vy)\mid d_g(\vx,\vy)\le \varepsilon\}\subset \{(\vx,\vy)\mid \underline{g}|\vx-\vy|\le\varepsilon\},
\end{equation}then the support of $\eta_{\varepsilon}(|\vx-\vy|)$ for each $\vx$ is the subset of $B_{ \frac{\varepsilon}{\lgg}}(\vx)$.
We know that \begin{equation}
    \mathbb{E} f(\vx, \vy)=\sigma_\eta I_\varepsilon[u]
\end{equation}and \begin{align}
    \operatorname{Var}f(\vx, \vy)
& \leq \int_\Omega \int_\Omega \eta_{\varepsilon}(|\vx-\vy|)^2\left(\frac{u(\vx)-u(\vy)}{\varepsilon}\right)^4 \rho(\vx) \rho(\vy) \D  \vx \D  \vy \\\notag& \leq \int_\Omega \int_\Omega \eta_{\varepsilon}(|\vx-\vy|)^2\operatorname{Lip}(u)^4\left(\frac{|\vx-\vy|}{\varepsilon} \right)^4\rho(\vx) \rho(\vy) \D  \vx \D  \vy \\\notag
& \leq C \operatorname{Lip}(u)^4 \varepsilon^{-2 d} \int_\Omega \int_{B_{ \frac{\varepsilon}{\lgg}}(\vx)\cap \Omega} \left(\frac{|\vx-\vy|}{\varepsilon} \right)^4\rho(\vx) \rho(\vy)\D  \vx \D  \vy \\\notag
& \leq C(\lgg) \operatorname{Lip}(u)^4 \varepsilon^{-d}.
\end{align}
Applying Bernstein's inequality for U-statistics yields
\[
\mathbb{P}\left(\left|U_n-\sigma_\eta I_{\varepsilon}(u)\right| \geq \operatorname{Lip}(u)^2 \lambda\right) \leq 2 \exp \left(-c n \varepsilon^d \lambda^2\right)
\]
for any $0<\lambda \leq 1$. Since $E_n[u]=\frac{n-1}{n \sigma_\eta} U_n$ we have
\[
\begin{aligned}
\left|E_n[u]-I_{\varepsilon}(u)\right| & =\frac{1}{\sigma_\eta}\left|\left(1-\frac{1}{n}\right) U_n-\sigma_\eta I_{\varepsilon}(u)\right| \\
& =\frac{1}{\sigma_\eta}\left|\left(1-\frac{1}{n}\right)\left(U_n-\sigma_\eta I_{\varepsilon}(u)\right)-\frac{\sigma_\eta}{n} I_{\varepsilon}(u)\right| \\
& \leq \frac{1}{\sigma_\eta}\left|U_n-\sigma_\eta I_{\varepsilon}(u)\right|+\frac{1}{n} I_{\varepsilon}(u) .
\end{aligned}
\]

Since $u$ is Lipschitz we have \begin{align}
    I_\varepsilon[u]=&\frac{1}{\sigma_\eta\varepsilon^d}\int_{\Omega}\int_{\Omega}\eta\left( \frac{d_g(\vx,\vy)}{\eps}\right)\frac{(u(\vx)-u(\vy))^2}{\varepsilon^2}\rho(\vx)\rho(\vy)\D  \vy\D  \vx\notag\\\le &\frac{1}{\sigma_\eta\varepsilon^d}\int_{\Omega}\int_{\Omega}\eta\left( \frac{d_g(\vx,\vy)}{\eps}\right)\frac{\operatorname{Lip}(u)^2|\vx-\vy|^2}{\varepsilon^2}\rho(\vx)\rho(\vy)\D  \vy\D  \vx\notag\\=&\frac{1}{\sigma_\eta\varepsilon^d}\int_{\Omega}\int_{B_{ \frac{\varepsilon}{\lgg}}(\vx)\cap\Omega}\eta\left( \frac{d_g(\vx,\vy)}{\eps}\right)\frac{\operatorname{Lip}(u)^2|\vx-\vy|^2}{\varepsilon^2}\rho(\vx)\rho(\vy)\D  \vy\D  \vx\notag\\\le &\frac{C(\lgg)\operatorname{Lip}^2(u)}{\sigma_\eta}
\end{align} Therefore we can apply the result of Bernstein (Lemma \ref{lem:bernstein}) above to obtain that
$$
\left|E_{n}(u)-I_{\varepsilon}(u)\right| \leq C(\lgg) \operatorname{Lip}(u)^2\left(\delta+\frac{1}{n}\right)
$$
holds with probability at least $1-2 \exp \left(-c n \varepsilon^d \delta^2\right)$, which completes the proof.
\end{proof}

\section{Neural Network Architecture and Training Details}\label{app:nn_details}

The neural network consists of three fully connected (linear) layers with intermediate SiLU (Sigmoid-weighted Linear Unit) activations, followed by a final Softplus activation to ensure non-negativity of the output. The input is first passed through a fully connected (linear) layer that maps the 3 input features to 8 hidden units. This is followed by a SiLU activation function, which introduces nonlinearity while preserving smooth gradients. The output is then passed through a second fully connected layer that also has 8 hidden units, again followed by a SiLU activation. A final linear layer maps the hidden representation to a single scalar output. To ensure the predicted diffusivity is strictly positive—a necessary condition for physical consistency—the final output is passed through a Softplus activation function.

The SiLU activation is defined as \( \mathrm{SiLU}(x) = x \cdot \sigma(x) \), where \( \sigma(x) \) is the sigmoid function. The Softplus activation is given by \( \mathrm{Softplus}(x) = \ln(1 + e^x) \), which guarantees strictly positive outputs.

The choice of this compact architecture was deliberate: when training on the 227 local FC values extracted from 68 brain parcels, more complex networks consistently led to overfitting. Since increasing the dataset size was not possible, we reduced the architecture instead, which allowed us to avoid overfitting and ensure stable generalization.

The neural network is trained using a dataset derived from functional connectivity (FC) values between nearby brain parcels. During training, $k$-fold cross-validation was employed to assess generalizability and robustness of the learned diffusivity field. Figure~\ref{fig:cv-result} shows the average training and validation losses obtained from 5-fold cross-validation using Eq. \eqref{eq:loss_function_modified}.



\renewcommand\thefigure{\thesection.\arabic{figure}} 
\setcounter{figure}{0} 
\begin{figure}[ht]
\centering
\includegraphics[width=0.85\textwidth]{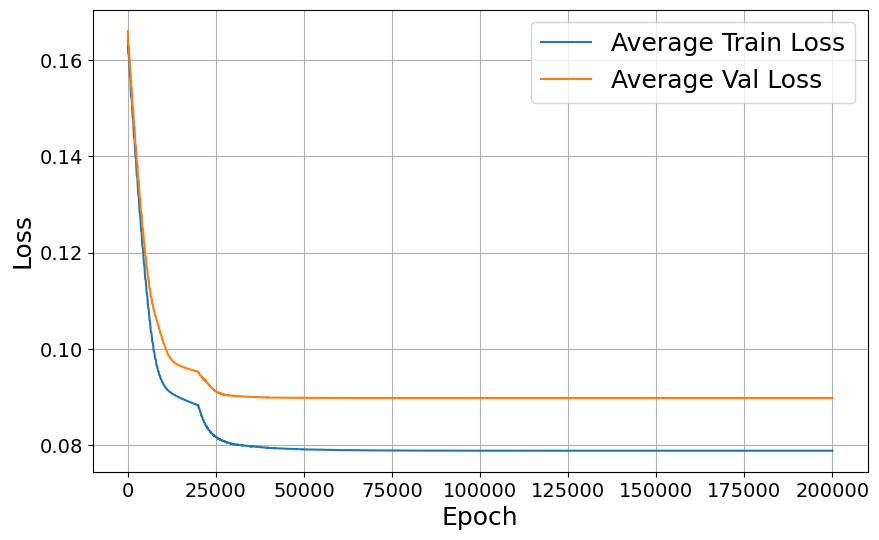}
\caption{Average training and validation losses across 5-fold cross-validation using the neural network model.}
\label{fig:cv-result}
\end{figure}
\end{document}